\documentclass{article}
\usepackage{amsmath, amssymb, amsthm, stackrel, enumerate}
\usepackage{titlesec, tikz,subcaption,pgf}
\usepackage{tikz-cd}
\usetikzlibrary{arrows}

\makeatletter
\newcommand\xleftrightarrow[2][]{%
	\ext@arrow 9999{\longleftrightarrowfill@}{#1}{#2}}
\newcommand\longleftrightarrowfill@{%
	\arrowfill@\leftarrow\relbar\rightarrow}
\makeatother

\newcommand{\rar}{\rightarrow}

\newcommand{\C}{\mathbb{C}}

\newcommand{\cat}[1]{\mathcal{#1}}

\newcommand{\Z}{\mathbb{Z}}

\newcommand{\Aut}{\textnormal{Aut}}

\newcommand{\id}{\textnormal{id}}

\newcommand{\End}{\textnormal{End}}

\newcommand{\Id}{\textnormal{Id}}

\newcommand{\Vect}{\mathbf{Vect}}

\newcommand{\tn}[1]{\textnormal{#1}}

\newcommand{\svec}{\mathbf{sVect}}

\newcommand{\cattens}[1]{\mathop{\boxtimes}\limits_{#1}}

\newcommand{\lincat}{\mathbf{LinCat}}

\newcommand{\Rep}{\mathbf{Rep}}
\newcommand{\lcat}[1]{\cat{#1}\lincat}
\newcommand{\bxc}[1]{\cattens{\cat{#1}}}

\newcommand{\raru}[1]{\mathop{\rar}\limits^{#1}}

\newcommand{\bx}{\boxtimes}

\newcommand{\dcentcat}[1]{\cat{Z}(\cat{#1})}
\newcommand{\symt}{\mathop{\otimes}\limits_{s}}
\newcommand{\cont}{\mathop{\otimes}\limits_{c}}

\newcommand{\ZAXBT}{\dcentcat{A}\tn{-}\mathbf{XBF}}
\newcommand{\BFCA}{\mathbf{BFC}/\cat{A}}
\newcommand{\GXBT}{G\tn{-}\mathbf{XBF}}
\newcommand{\sGXBT}{(G,\omega)\tn{-}\mathbf{XBF}}

\newcommand{\Lin}{\mathbf{DE}}
\newcommand{\DE}{\mathbf{DE}}

\usepackage{xr}
\newtheorem{thm}{Theorem}
\newtheorem{cor}[thm]{Corollary}
\newtheorem{lem}[thm]{Lemma}
\newtheorem{prop}[thm]{Proposition}
\theoremstyle{definition}
\newtheorem{ex}[thm]{Example}
\newtheorem{df}[thm]{Definition}
\newtheorem{rmk}[thm]{Remark}
\newtheorem{defprp}[thm]{Definition-Propostion}

\newtheorem{conj}[thm]{Conjecture}
\newtheorem{notation}[thm]{Notation}

\title{Drinfeld Centre-Crossed Braided Categories}
\author{Thomas A. Wasserman}

\begin{document}
	\maketitle
	\begin{abstract}
		We introduce, for a symmetric fusion category $\cat{A}$ with Drinfeld centre $\dcentcat{A}$, the notion of $\dcentcat{A}$-crossed braided tensor category. These are categories that are enriched over $\dcentcat{A}$ equipped with a symmetric tensor product, while being braided monoidal with respect to the usual tensor product on $\dcentcat{A}$. 
		
		In the Tannakian case where $\cat{A}=\Rep(G)$ for a finite group $G$, the 2-category of $\dcentcat{A}$-crossed braided categories is shown to be equivalent to the 2-category of $G$-crossed braided tensor categories. A similar result is established for the super-Tannakian case where $\cat{A}$ is the representation category of a finite super-group.
	\end{abstract}
	\tableofcontents

\section{Introduction}

In this article, we introduce the novel notion of Drinfeld centre crossed braided tensor categories. In a previous paper \cite{Wasserman2017a}, the author has shown that the Drinfeld centre $\dcentcat{A}$ of a symmetric fusion category $\cat{A}$ is lax 2-fold monoidal for the convolution and symmetric tensor products. In this paper, we will use the extra structure this 2-fold product gives to categories enriched over $(\dcentcat{A},\otimes_s)$ to define a novel notion of monoidal structure for such categories. This monoidal structure factors, on hom-objects, through the convolution tensor product of such categories. We will refer to this as a $\dcentcat{A}$-crossed monoidal structure (Definition \ref{ZXzacrossedtensordef}). Additionally, we spell out what it means for such a monoidal structure to be braided (Definition \ref{ZXcrossedbraidingdef}). Such categories form a 2-category $\ZAXBT$, and we will show that the $(\dcentcat{A},\otimes_s)$-enriched cartesian product $\cattens{s}$ is a symmetric monoidal structure on this 2-category.

In his book on Homotopy Quantum Field Theory \cite[Chapter VI]{Turaev2010}, Turaev defined the notion of $G$-crossed braided fusion category (see Definition \ref{ZXgcrossedbraided} here). By results of Deligne \cite{Deligne1990,Deligne2002}, any symmetric fusion category is braided tensor equivalent to the representation category of either a finite group $G$ or a finite super-group $(G',\omega)$, referred to as the Tannakian and super-Tannakian cases respectively. We will show that, for $\cat{A}\cong \Rep(G)$, a $\dcentcat{A}$-crossed braided category gives rise to such a $G$-crossed braided category. Furthermore, for $\cat{A}$ super-Tannakian, we define the notion of a super $(G',\omega)$-crossed braided category (Definition \ref{ZXsgcrossedbraided}), and show that, similarly, a $\dcentcat{A}$-crossed braided category gives rise to a super $(G',\omega)$-crossed braided category. These two results constitute Proposition \ref{ZXbarbariscrossedbraided}. We show that this construction, denoted by $\overline{ \overline{(-)}}$, has an inverse $\mathbf{Fix}$, this is the main Theorem \ref{ZXzaxbtgxbteqv} of this paper. Furthermore, the 2-category $\GXBT$ (or $\sGXBT$) of (super) $G$-crossed braided categories carries a symmetric monoidal structure given by the degreewise product $\cattens{G}$, and we show that these mutually inverse constructions are symmetric monoidal 2-functors.

This relationship to $G$-crossed braided tensor categories tells us that $\dcentcat{A}$-braided tensor categories are a coordinate invariant repackaging of the notion of $G$-crossed braided category, that additionally allows us to treat the Tannakian and super-Tannakian cases on equal footing. This is useful in related work \cite{Wassermanc}, were we take a braided fusion category containing $\cat{A}$ and produce from this a braided object enriched over $\dcentcat{A}$. It turns out that this braided object is a $\dcentcat{A}$-crossed braided tensor category. This construction, which we will denote by $\underleftarrow{\underline{(-)}}$ is then used to define the so-called reduced tensor product. This reduced tensor product is a symmetric monoidal structure $\cattens{\tn{red}}$ on the 2-category $\BFCA$ of braided fusion categories containing $\cat{A}$. 

The construction that takes a braided fusion category containing $\cat{A}$ to a $\dcentcat{A}$-crossed braided fusion category and its inverse, together with the functors from the main theorem of this paper, factor the well-known mutually inverse constructions of equivariantisation $\mathbf{Eq}$ and de-equivariantisation $\mathbf{De-Eq}$ \cite{Drinfeld2009} into two steps. Critically, the step corresponding to the main Theorem of this paper requires Tannaka duality, while the other does not. The situation is summarised by the following commutative triangle of mutually inverse symmetric monoidal equivalences:
\begin{center}
	\begin{tikzcd}
	&			(\ZAXBT,\cattens{s}) \arrow[ddl,shift left,"\mathbf{DeEnrich}"]	\arrow[ddr,shift left,"\overline{ \overline{(-)}}"]		&\\
	& & \\
	(\BFCA,\cattens{\tn{red}}) \arrow[ruu,shift left,"\underleftarrow{\underline{(-)}}"] \arrow[rr,shift left,"\mathbf{De-Eq}"] & & (\GXBT,\cattens{G})\arrow[ll,shift left, "\mathbf{Eq}"]\arrow[luu,shift left,"\mathbf{Fix}"],
	\end{tikzcd}
\end{center}
in the case that $\cat{A}=\Rep(G)$, and $\mathbf{DeEnrich}$ denotes the inverse to $\underleftarrow{\underline{(-)}}$. A similar diagram exists for the super-Tannakian case. The leftmost edge of this diagram makes no reference to Tannaka duality. This article concerns itself with the rightmost edge of this triangle.

The structure of the present paper is as follows. We will first, in Section \ref{ZXdccbcsect}, develop the theory of $\dcentcat{A}$-crossed braided categories. Section \ref{ZXdcxbsupertannasect} is devoted to stating the definition of a (super) $G$-crossed braided category, and proving Theorem \ref{ZXbarbariscrossedbraided}, which tells us $\dcentcat{A}$-crossed braided categories give rise to (super)-$G$-crossed braided categories. After this, in Section \ref{ZXgxtozax}, we construct an inverse to this construction, and prove in Section \ref{ZXequivgxzx} our main Theorem \ref{ZXzaxbtgxbteqv}.

\subsubsection*{Acknowledgements}
The author would like to thank Chris Douglas and Andr\'e Henriques for their guidance during the project. An earlier version of the manuscript benefited greatly from Michael M\"uger's and Ulrike Tillmann's comments. 

The research presented here was made possible through financial support from the Engineering and Physical Sciences Research Council for the United Kingdom, the Prins Bernhard Cultuurfonds, the Hendrik Mullerfonds, the Foundation Vrijvrouwe van Renswoude and the Vreedefonds. The author is supported by the by the Centre for
Symmetry and Deformation at the University of Copenhagen (CPH-SYM-DNRF92) and Nathalie Wahl's European Research Council Consolidator Grant (772960).

\section{$\dcentcat{A}$-crossed braided categories}\label{ZXdccbcsect}
We now set up the theory of $\dcentcat{A}$-crossed braided categories.

\begin{notation}\label{ZXnotation2fold}
We will write $\dcentcat{A}_s$ for the Drinfeld centre of a symmetric ribbon fusion category $\cat{A}$ equipped with the symmetric tensor product from  \cite{Wasserman2017}.
In \cite{Wasserman2017a} it was shown that $(\dcentcat{A},\otimes_c,\otimes_s)$ is bilax 2-fold tensor. Following the notation there, we will denote the compatibility morphisms by $(\eta,u_0,u_1,u_2)$ for the lax direction and $(\zeta, v_0,v_1,v_2)$ for the oplax direction. That is, we have morphisms:
\begin{align}
\eta_{c,c',d,d'}\colon (c\otimes_c c')\otimes_s (d\otimes_c d') &\leftrightarrows (c\otimes_s d)\otimes_c(c'\otimes_s d'):\!\zeta_{c,d,c',d'}\label{ZXetazeta}\\
u_0\colon  \mathbb{I}_s &\leftrightarrows \mathbb{I}_c :\! v_0\label{ZXu0v0def}\\
u_1\colon \mathbb{I}_c \otimes_s \mathbb{I}_c & \leftrightarrows \mathbb{I}_c:\!v_2\label{ZXu1v2def}\\
u_2\colon \mathbb{I}_s & \leftrightarrows \mathbb{I}_s \otimes_c \mathbb{I}_s:\!v_1,\label{ZXv1u2def}
\end{align}
where $\mathbb{I}_c$ and $\mathbb{I}_s$ denote the monoidal units for $\otimes_c$ and $\otimes_s$, respectively. These morphisms satisfy the additional equations:
$$
\eta\circ\zeta = \tn{id},\quad u_0\circ v_0=\tn{id}, \quad u_1\circ v_2=\tn{id},\quad v_1\circ u_2=\tn{id},
$$ 
and the morphisms $u_1$ and $v_2$ are inverse isomorphisms. These additional relations make $\dcentcat{A}$ into what is called a strongly inclusive bilax 2-fold monoidal tensor in \cite{Wasserman2017a}. The morphisms above are also compatible with the braiding $\beta$ for $\otimes_c$ and the symmetry $s$ for $\otimes_s$ in the sense that
\begin{equation}\label{ZXbraidingcompatible}
\eta_{c',c,d',d}\circ(\beta_{c,c'}\otimes_s\beta_{d,d'})=\beta_{c\otimes_s d,c'\otimes_s d'}\circ\eta_{c,c',d,d'},
\end{equation}
and similar conditions are satisfied. This is spelled out further in \cite[Definition \ref{DCbraideddef}]{Wasserman2017a}.

\end{notation}

\subsection{$\dcentcat{A}_s$-enriched categories}
In this section we will develop the theory of $\dcentcat{A}_s$-enriched and tensored categories. On top of some basic lemmas, we will need some notions the theory of abelian categories, such as idempotent completeness, simple objects and semi-simplicity. We will start by recalling the relevant notions for $\Vect$-enriched categories, and then generalise these to $\dcentcat{A}$-enriched categories.

\subsubsection{Linear categories}
If a category is enriched in and tensored over the category of vector spaces, it is an additive category. We will call such a category \emph{linear} if it is additionally abelian for this additive structure. An abelian category is called \emph{semi-simple} if every object is isomorphic to a direct sum of finitely many simple objects. We additionally need the notion of idempotent completeness. Recall that an endomorphism $f$ of an object $d$ in a category $\cat{D}$ is called \emph{idempotent} if $f^2=f$. 

\begin{df}
	A category $\cat{D}$ is called \emph{idempotent complete} if for every idempotent $f\in \End_{\cat{D}}(d)$ there is an object $d_f$ (a \emph{subobject associated to $f$}) together with a monic $i_f:d_f\hookrightarrow d$ and an epi $p_f:d\twoheadrightarrow d_f$ such that $i_f\circ p_f=f$.
\end{df}

\subsubsection{The associated $\Vect$-category}
To be able to use notions from the theory of abelian categories, it is useful to observe that associated to any $\dcentcat{A}_s$-enriched category, there is a $\Vect$-enriched category.

\begin{df}\label{ZXlin}
	Let $\cat{C}$ be a $\dcentcat{A}_s$-enriched and tensored category, then its associated $\Vect$-enriched and tensored category $\Lin(\cat{C})$ is the category obtained by changing basis (Proposition \ref{ZXchangebasisobjects}) along the functor
	$$
	\dcentcat{A}_s(\mathbb{I}_s,-):\dcentcat{A}\rar \Vect.
	$$
\end{df}

\begin{df}\label{ZXZAlindef}
	A $\dcentcat{A}_s$-enriched and tensored category is called $\dcentcat{A}$-linear if its associated $\Vect$-enriched category $\Lin(\cat{C})$ is linear, and semi-simple if $\Lin(\cat{C})$ is semi-simple.
\end{df}

\begin{rmk}
	In our context of enriched category theory, where the enriching category is abelian, care has to be taken in interpreting the equation $f^2=f$. The appropriate interpretation is as an equality between \emph{members} (for details on these see \cite[Page 204]{MacLane}).
	
	One can show that idempotent completeness of $\Lin(\cat{C})$ and of $\cat{C}$ are equivalent for any category $\cat{C}$ enriched in $\dcentcat{A}_s$, or more generally any category enriched in and tensored over a monoidal category.
\end{rmk}

In the second half of this article, we will mainly be concerned with the following $\dcentcat{A}_s$-enriched categories:

\begin{df}
	A \emph{FD $\dcentcat{A}_s$-category} is an idempotent complete semi-simple $\dcentcat{A}_s$-linear category with a finite number of isomorphism classes of simple objects.
\end{df}

\subsubsection{$\dcentcat{A}$-tensoring}
We will need the following fact about the interaction between the other tensor product $\otimes_c$ on $\dcentcat{A}$ and the $\dcentcat{A}_s$-enriched and tensored structure.

\begin{prop}\label{ZXatensequivcond}
	Let $\cat{K}$ be an $\dcentcat{A}_s$-enriched and tensored category, and denote its $\dcentcat{A}_s$-tensoring by $\cdot$. Then we have for $a\in\cat{A}\subset\dcentcat{A}$:
	\begin{equation}\label{ZXtensequivisos}
	a\otimes_c \cat{K}(-,k)\mathop{\Rightarrow}\limits^{\cong} \cat{K}(-,(a\otimes_c \mathbb{I}_s)\cdot k).
	\end{equation}
\end{prop}
\begin{proof}
	By \cite[Lemma \ref{STconvvssymunit}]{Wasserman2017}, for $z\in \dcentcat{A}$ and $a \in \cat{A}$, we have:
	$$
	a \otimes_c z \cong	(a\otimes_c \mathbb{I}_s) \otimes_s z.
	$$
	By Lemma \ref{ZXaintohom}, this means that we have for all $k,k'\in\cat{K}$:
	$$
	a \otimes_c \cat{K}(k,k')\cong \cat{K}(k, (a \otimes_c \mathbb{I}_s)\cdot k'),
	$$
	and this isomorphism is natural in $a$, $k$ and $k'$.
\end{proof}

\subsubsection{The symmetric monoidal bicategory of $\dcentcat{A}$-enriched and tensored categories}
\begin{df}\label{ZXzalincat}
	The \emph{2-category $\lcat{Z(A)}$ of $\dcentcat{A}_s$-linear idempotent complete categories} has morphisms $\dcentcat{A}_s$-enriched functors $F\colon \cat{K}\rar \cat{K}'$ which respect the $\dcentcat{A}_s$-tensoring, and 2-morphisms $\dcentcat{A}_s$-enriched natural tranformations $\eta$ satisfying $\eta_{ak}=\id_a \eta_k$.

\end{df}

As this is a 2-category of categories enriched over a symmetric category, it comes equipped with a symmetric monoidal structure, see Definition \ref{ZXaproddef}.

\begin{df}\label{ZXcattenssdef}
	We will denote by $\cattens{s}$ the \emph{Cauchy completion of the enriched Cartesian product of $\dcentcat{A}_s$-enriched and tensored categories}. (The notion of Cauchy completion is  defined in Definition \ref{ZXcauchycompl}.)
\end{df}

In analogy with the situation for $\lincat$, see \cite{Bartlett2015a}, we expect:

\begin{conj}
	The FD $\dcentcat{A}$-categories are the fully dualisable objects in $\lcat{Z(A)}$.
\end{conj}

\subsubsection{The associated $\cat{A}$-enriched category}
As the reader might expect, $\dcentcat{A}_s$-enriched categories are intimately related to $\cat{A}$-enriched categories. We spell out how to produce an $\cat{A}$-enriched category from a $\dcentcat{A}_s$-enriched category in this section. This construction will also be an important ingredient in the second half of the paper.

\begin{df}\label{ZXatensorforzaenr}
	Let $\cat{K}$ be a $\dcentcat{A}_s$-enriched category. Then the \emph{associated $\cat{A}$-enriched category} $ \overline{\cat{K}}$ for $\cat{K}$ is the category obtained by applying the forgetful functor $\mathbf{Forget}:\dcentcat{A}_s\rar \cat{A}$.
\end{df}

\begin{lem}
	$ \overline{\cat{K}}$ is indeed an $\cat{A}$-enriched category.
\end{lem}

\begin{proof}
	By \cite[Proposition \ref{STforgetislaxmon}]{Wasserman2017}, the forgetful functor is lax monoidal, so this is a direct consequence of Proposition \ref{ZXchangebasisobjects}.
\end{proof}

Analogous to Definition \ref{ZXlin}, we define:
\begin{df}
	Let $\cat{C}$ be a $\cat{A}$-enriched and tensored category, then its \emph{associated $\Vect$-enriched and tensored category} $\Lin(\cat{C})$ is the category obtained by change of basis along the functor $\cat{A}(\mathbb{I}_\cat{A},-)\colon \cat{A}\rar \Vect$.
\end{df}

The abuse of notation in this definition is justified by the following:

\begin{lem}\label{ZXdeenrichequivmeth}
	The functors $\dcentcat{A}(\mathbb{I}_s,-)$ and $\cat{A}(\mathbb{I}_\cat{A},-)\circ\mathbf{Forget}$ from $\dcentcat{A}$ to $\Vect$ are canonically isomorphic.
\end{lem}

\begin{proof}
	Let $c=(a,\beta)\in\dcentcat{A}$ and denote by $s$ the symmetry in $\cat{A}$. By definition of $\dcentcat{A}$, the hom space $\dcentcat{A}(\mathbb{I}_s,(a,\beta))$ is the subspace of $\cat{A}(\mathbf{Forget}(\mathbb{I}_s),a)$ of those morphisms that intertwine the half-braiding $\gamma$ on $\mathbb{I}_s$ and $\beta$. Using the adjunction $\cat{A}(\mathbf{Forget}(\mathbb{I}_s),a)\cong\cat{A}(\mathbf{Forget}(\mathbb{I}_s)a^*,\mathbb{I}_\cat{A})$, this intertwining condition translates to picking out those morphisms on $\mathbf{Forget}(\mathbb{I}_s)a^*$ for which pre-composing with $s\otimes\beta$ is the same as pre-composing with $\gamma \otimes s$. In other words, we are picking out those morphisms which factor through the maximal subobject of $\mathbf{Forget}(\mathbb{I}_s)a^*$ on which $s\otimes \beta$ and $\gamma \otimes s$ agree. But by \cite[Lemma \ref{STkeyprop}]{Wasserman2017}, this is the object of $\cat{A}$ underlying $\mathbb{I}_s\otimes_sa^*$, which, as $\mathbb{I}_s$ is the unit for $\otimes_s$, is canonically isomorphic to $a^*$. So we have established:
	$$
	\dcentcat{A}(\mathbb{I}_s,(a,\beta))\cong \cat{A}(a^*,\mathbb{I}_\cat{A})\cong\cat{A}(\mathbb{I}_\cat{A},a),
	$$
	canonically.
\end{proof}

By Proposition \ref{ZXlaxnattrafotofunctor}, we therefore have:

\begin{cor}\label{ZXlinislinbar}
	Let $\cat{K}$ be a $\dcentcat{A}$-enriched and tensored category. Then
	$$
	\Lin(\cat{K})\cong\Lin(\overline{ \cat{K}}).
	$$
\end{cor}

The assignment of taking $\dcentcat{A}$-categories to $\cat{A}$-categories extends to a 2-functor:
$$
\overline{(-)}\colon \lcat{Z(A)}\rar \lcat{A},
$$
where $\lcat{A}$ was defined in Definition \ref{ZXalincatdef}.

The 2-functor $\overline{(-)}$ will in fact be bilax symmetric monoidal, as the following Lemma indicates. Let $c=(a,\beta)$ and $c'=(a',\beta')$ be objects of $\dcentcat{A}$. Recall from \cite{Wasserman2017} that the symmetric tensor product $c\otimes_s c'$ is defined in terms of the image of  a subobject $c\otimes_\Pi c'$, associated to an idempotent $\Pi_{c,c'}$ on $c\otimes_c c'$, under the forgetful map. This subobject of $c\otimes_cc'$ comes with an inclusion and a projection. We will denote the images of these morphisms under the forgetful functor by 
$$
i_{c,c'}\colon \mathbf{Forget}(c\otimes_\Pi c') \rar a a' 
$$
and  
$$
p_{c,c'}\colon a a'\rar  \mathbf{Forget}(c\otimes_\Pi c') ,
$$
respectively. We can use these to examine how the assignment $\cat{K}\mapsto\overline{ \cat{K}}$ interacts with the enriched Cartesian products:

\begin{lem}\label{ZXforgetonsymprod}
	For $\cat{K}$ and $\cat{L}$ be $\dcentcat{A}_s$-enriched categories, we have that the map
	$$
	\overline{P}\colon  \overline{\cat{K}}\cattens{\cat{A}}\overline{\cat{L}} \rar \overline{\cat{K}\cattens{s}\cat{L}},
	$$
	which acts the identity on objects and as $p_{-,-}$ on hom-objects, is an $\cat{A}$-enriched functor. It has a right-sided inverse:
	$$
	\overline{I}\colon\overline{\cat{K}\cattens{s}\cat{L}} \rar \overline{\cat{K}}\cattens{\cat{A}}\overline{\cat{L}},
	$$
	which is also the identity on objects, and acts as $i_{-,-}$ on hom-objects.
\end{lem} 

\begin{proof}
	The first part is a direct consequence of Proposition \ref{ZXchangebasisproduct}. That $\overline{I}$ is a right-sided inverse follows from $p\circ i=\id$ and Proposition \ref{ZXchangebasisispseudofunct}.
\end{proof}

\subsubsection{The Neutral Subcategory}
In the section above, we considered the forgetful functor $\dcentcat{A}\rar \cat{A}$. As $\cat{A}$ is symmetric, we also have a braided monoidal (for either monoidal structure on $\dcentcat{A}$) functor $\cat{A}\hookrightarrow \dcentcat{A}$. This allows us to view any $\cat{A}$-enriched category as a $\dcentcat{A}_s$-enriched category. To test whether a $\dcentcat{A}_s$-enriched category $\cat{K}$ is in the image of this 2-functor, we introduce the \emph{neutral subcategory} $\cat{K}_\cat{A}$. The definition is such that $\cat{K}=\cat{K}_\cat{A}$ if and only if $\cat{K}$ comes from a $\cat{A}$-enriched category. We define $\cat{K}_\cat{A}$ by:

\begin{df}\label{ZXneutralsubcat}
	Let $\cat{K}$ be a $\dcentcat{A}_s$-enriched category. Then the full subcategory on objects $k\in \cat{K}$ for which the Yoneda embedding factors as:
	$$
	\cat{K}(-,k)\colon  \cat{K}^{\tn{op}} \rar \cat{A} \hookrightarrow \dcentcat{A},
	$$
	is called the \emph{neutral subcategory} of $\cat{K}$ and will be denoted by $\cat{K}_\cat{A}$.
\end{df}

There is another characterisation of the neutral subcategory:

\begin{lem}\label{ZXneutralsubcatendochar}
	Let $k$ be an object of a $\dcentcat{A}_s$-enriched category $\cat{K}$. This object is in $\cat{K}_\cat{A}$ if and only if the endomorphism object $\cat{K}(k,k)$ of $k$ is an object of $\cat{A}\subset \dcentcat{A}$.
\end{lem}

\begin{proof}
	The ``only if'' direction is obvious. For the other direction, observe that for any $k'\in \cat{K}$, we have an automorphism of $\cat{K}(k',k) $ given by the composite
	$$
	\cat{K}(k',k) \cong \mathbb{I}_s \otimes_s \cat{K}(k',k) \xrightarrow{\id_k} \cat{K}(k,k)\otimes_s \cat{K}(k',k) \xrightarrow{\circ}	\cat{K}(k',k).
	$$
	Assuming that $\cat{K}(k,k)\in \cat{A}\subset \dcentcat{A}$, we see, by \cite[Proposition \ref{STsymtenswitha}]{Wasserman2017}, that this automorphism factors through an object of $\cat{A}$, and hence that $\cat{K}(k',k)$ is an object of $\cat{A}$. 
\end{proof}

By \cite[Proposition \ref{STsymtenswitha}]{Wasserman2017}, the subcategory $\cat{A}\subset\dcentcat{A}_s$ annihilates its complement. This translates to the following for the product $\cattens{s}$ of $\dcentcat{A}_s$-enriched and tensored categories from Definition \ref{ZXcattenssdef}:

\begin{prop}\label{ZXtensofneutralpart}
	Let $\cat{K}$ and $\cat{L}$ be $\dcentcat{A}_s$-enriched and tensored categories. Then:
	$$
	(\cat{K}\cattens{s}\cat{L})_\cat{A}\cong\cat{K}\cattens{s}\cat{L}_\cat{A}\cong\cat{K}_\cat{A}\cattens{s}\cat{L}\cong\overline{\cat{K}_\cat{A}}\cattens{\cat{A}}\overline{\cat{L}_\cat{A}},
	$$
	where we view the $\cat{A}$-enriched and tensored category on the right as $\dcentcat{A}_s$-enriched and tensored category by using the symmetric strong monoidal inclusion functor $\cat{A}\hookrightarrow \dcentcat{A}$.
\end{prop}

\begin{proof}
	We prove the equivalence between the left- and rightmost categories first. The category on the left hand side is a full subcategory of $\cat{K}\cattens{s}\cat{L}$, we define a functor $\overline{\cat{K}_\cat{A}}\cattens{\cat{A}}\overline{\cat{L}_\cat{A}}\rar \cat{K}\cattens{s}\cat{L}$ by using the inclusion $\cat{A}\hookrightarrow \dcentcat{A}$ on hom-objects and claim its essential image is $(\cat{K}\cattens{s}\cat{L})_\cat{A}$.
	
	To show this functor is essentially surjective onto $(\cat{K}\cattens{s}\cat{L})_\cat{A}$, let $k\boxtimes l$ be an object of $(\cat{K}\cattens{s}\cat{L})_\cat{A}$. For this $k$ and $l$, denote the summands contained in $\cat{K}_\cat{A}$ and $\cat{L}_\cat{A}$ by $k_\cat{A}$ and $l_\cat{A}$, respectively. We claim that:
	$$
	k_\cat{A}\boxtimes l_\cat{A}\cong k\boxtimes l.
	$$
	We will show this by examining their Yoneda embeddings. The object on the left hand side has Yoneda embedding $\cat{K}(-,k_\cat{A})\otimes_s \cat{L}(-,l_\cat{A})$ whereas the right hand side has $\cat{K}(-,k)\otimes_s \cat{L}(-,l)$, and we claim that the image of the inclusions $i_k:k_\cat{A}\hookrightarrow k$ and $i_l:l_\cat{A}\hookrightarrow l$ is a natural isomorphism between these functors. Let $k'\boxtimes l'$ be an object of $\cat{K}\cattens{s}\cat{L}$, then we want to show that
	$$
	\cat{K}(k',k_\cat{A}) \otimes_s \cat{L}(l',l_\cat{A}) \xrightarrow{(i_k)_*\otimes_s (i_l)_*} \cat{K}(k',k) \otimes_s \cat{L}(l',l)
	$$
	is an isomorphism. By \cite[Proposition \ref{STsymtenswitha}]{Wasserman2017}, the symmetric tensor product of two objects in $\dcentcat{A}$ is a non-zero object of $\cat{A}$ if and only if the $\cat{A}$-summands of these objects are non-zero, and the part that lies in $\cat{A}$ is the product of these summands. The objects $\cat{K}(k',k_\cat{A})$ and  $\cat{L}(l',l_\cat{A})$ are the $\cat{A}$-summands of $\cat{K}(k',k_\cat{A})$ and $\cat{L}(l',l_\cat{A})$ respectively, so the claim follows. The same argument also establishes the functor is fully faithful.
	
	To see the other equivalences, note that the argument above also works if we only take the neutral summand of $k$ or $l$.
\end{proof}

In other words, the essential image of the symmetric monoidal 2-functor $\cat{A}\lincat\rar \dcentcat{A}_s\lincat$ is a tensor ideal for $\cattens{s}$. 

The inclusion $\cat{A}\rar \dcentcat{A}$ has a (two-sided) adjoint defined by taking simples in $\cat{A}$ to simples in $\cat{A}$ and simples in the complement of $\cat{A}$ in $\dcentcat{A}$ to $0$. This induces an adjoint 2-functor $\dcentcat{A}\lincat\rar \cat{A}\lincat$, which exposes $\cat{A}\lincat$ as a reflective subcategory of $\dcentcat{A}_s\lincat$. In particular, taking a $\dcentcat{A}_s$-linear category to its neutral subcategory is 2-functorial, as this is just the composition of this adjoint with inclusion.

\subsection{$\dcentcat{A}$-crossed categories}
Up to this point, we have concerned ourselves with spelling out elements of enriched category theory for the case where the enriching category is $\dcentcat{A}_s$ or a category related to it by a (lax) monoidal functor. We will now use the bilax 2-fold monoidal structure on $\dcentcat{A}$ from \cite{Wasserman2017a} to introduce the novel notion of a crossed tensor structure on a $\dcentcat{A}_s$-enriched category. Where normally an enriched monoidal category is defined as an enriched category with a monoidal structure that factors through the enriched cartesian product, we will here introduce the \emph{convolution product}\footnote{Named after the convolution product on $\Vect_{G}[G]$, which corresponds to $\otimes_c$ on $\dcentcat{A}$ under the Tannaka duality $\cat{A}\cong \Rep(G)$, see \cite{Wasserman2017}.} of $\dcentcat{A}_s$-enriched categories and define a crossed tensor product on these to be a monoidal structure that factors through this convolution product.

\subsubsection{Convolution Product of $\dcentcat{A}_s$-Enriched Categories}
\begin{df}\label{ZXconvprod}
	Let $\cat{K},\cat{L}$ be categories enriched over $\dcentcat{A}_s$. Then the \emph{convolution product $\cat{K}\cattens{c}\cat{L}$ of $\cat{K}$ and $\cat{L}$} is the Cauchy completion of the $(\dcentcat{A},\symt)$-enriched category with objects symbols $k\boxtimes l$ for $k\in\cat{K}$ and $l\in\cat{L}$, and hom-objects
	$$
	{\cat{K}\cattens{c}\cat{L}}(k\boxtimes l,k'\boxtimes l'):=\cat{K}(k,k')\cont\cat{L}(l,l').
	$$
	The composition is defined by the composite of the projection $\eta$ from Equation \eqref{ZXetazeta} and the compositions in $\cat{K}$ and $\cat{L}$. The identity morphisms are given by the composite
	\begin{equation}\label{ZXconvprodunit}
	\tn{I}_{k\boxtimes l}\colon \mathbb{I}_s\xrightarrow{u_2}  \mathbb{I}_s \cont \mathbb{I}_s \xrightarrow{\tn{I}_k\cont \tn{I}_l} \cat{K}(k,k)\cont\cat{L}(l,l),
	\end{equation}
	where $\tn{I}_k$ and $\tn{I}_l$ correspond to the identity morphisms on $k$ and $l$, respectively. The morphism $u_2$ was introduced in Equation \ref{ZXv1u2def}.
\end{df}

This is equivalenty the Cauchy completion of the category obtained by change of basis along the functor $\otimes_c\colon \dcentcat{A}_s\boxtimes\dcentcat{A}_s\rar \dcentcat{A}_s$ for the Deligne tensor product $\cat{K}\boxtimes\cat{L}$, and hence it is $\dcentcat{A}_s$-enriched by Proposition \ref{ZXchangebasisobjects}. Similarly to the situation for linear categories, the Cauchy completion ensures that, if the input categories are, then the resulting category is semi-simple with finitely many simples.

We also have:
\begin{cor}
	If $\cat{K}$ and $\cat{L}$ are $\dcentcat{A}_s$-enriched and tensored, then $\cat{K}\cattens{c}\cat{L}$ is $\dcentcat{A}_s$-tensored, with tensoring
	$$
	a (k\boxtimes l):= (ak\boxtimes l) \cong (k\boxtimes al).
	$$
\end{cor}

\begin{proof}
	This follows from Proposition \ref{ZXcarttensored}. 
\end{proof}

As long as we restrict our attention to categories that are $\dcentcat{A}_s$-tensored, the unit for the convolution product is $\cat{A}$ enriched over $\cat{A}\subset\dcentcat{A}_s$, denoted by $\cat{A}_\cat{Z}$, see Definition \ref{ZXaselfenriched}. If we drop the tensoring, the unit would become the one object category with endomorphism object for this single object $\mathbb{I}_c\in \dcentcat{A}$. This category is not $\dcentcat{A}_s$-tensored, and taking the free $\dcentcat{A}_s$-enriched and $\dcentcat{A}_s$-tensored category on this gives $\cat{A}_\cat{Z}$.

\begin{lem}\label{ZXunitconvprod}
	The convolution product of $\cat{A}_\cat{Z}$ with any $\dcentcat{A}_s$-enriched and tensored category $\cat{K}$ is equivalent to $\cat{K}$.
\end{lem}

\begin{proof}
	From Proposition \ref{ZXatensequivcond}, we get a functor
	\begin{align*}
	\cat{A}_\cat{Z}\cattens{c}\cat{K} &\rar \cat{K}\\
	a \boxtimes k &\mapsto (a\otimes_c\mathbb{I}_s)\cdot k,
	\end{align*}
	which on morphisms is the composite (using Equation \ref{ZXtensequivisos}, \cite[Lemma \ref{STconvvssymunit}]{Wasserman2017}, Lemma \ref{ZXaintohom} and \ref{ZXunitinternalhom}, and the adjunction for duals in $\cat{A}$):
	\begin{align*}
	\cat{A}_\cat{Z}(a,a')\otimes_c\cat{K}(k,k')	&\cong a^*\otimes_c\cat{A}_\cat{Z}(\mathbb{I}_\cat{A},a')\otimes_c\cat{K}(k,k')\\
	&\cong \cat{K}(k,(a^* \otimes_c a'\otimes_c\mathbb{I}_s)\cdot k')\\
	&\cong \cat{K}(k,(a^* \otimes_c \mathbb{I}_s)\cdot (a'\otimes_c\mathbb{I}_s)\cdot k')\\
	&\cong \cat{K}((a\otimes_c \mathbb{I}_s)\cdot k,(a'\otimes_c\mathbb{I}_s)\cdot k').
	\end{align*}
	This is fully faithful (induces isomorphisms on hom-objects) by construction, and seen to be essentially surjective by taking $a=\mathbb{I}_c$, hence an equivalence.
\end{proof}

\begin{df}\label{ZXbraidingfunctor}
	Let $\cat{K},\cat{L}$ be categories enriched over $\dcentcat{A}_s$. Then the \emph{braiding functor} 
	$$
	B\colon  \cat{K}\cattens{c}\cat{L}\rar \cat{L}\cattens{c}\cat{K}
	$$
	is given by $k\boxtimes l \mapsto l \boxtimes k$ on objects and by the braiding in $(\dcentcat{A},\otimes_c)$ on Hom-objects.
\end{df}

\begin{lem}
	$B$ indeed defines a $\dcentcat{A}_s$-enriched functor. Furthermore, this functor is an equivalence.
\end{lem}

\begin{proof}
	Viewing $\cat{K}\cattens{c}\cat{L}$ as coming from a change of basis (Proposition \ref{ZXchangebasisispseudofunct}) on $\cat{K}\boxtimes\cat{L}$ along $\otimes_c$ from $\dcentcat{A}_s\boxtimes \dcentcat{A}_s$ to $\dcentcat{A}_s$, we notice we can get $ \cat{L}\cattens{c}\cat{K}$ from a change of basis on $\cat{K}\boxtimes\cat{L}$ along $\otimes_c$ precomposed with the $\lincat$-switch functor in $\dcentcat{A}_s\boxtimes \dcentcat{A}_s\rar \dcentcat{A}_s\boxtimes \dcentcat{A}_s$. By definition, the braiding in $\dcentcat{A}_c$ gives a natural isomorphism between these two functors. Hence, by Proposition \ref{ZXlaxnattrafotofunctor}, if the braiding is lax monoidal with respect to $\otimes_s$, the braiding will induce the functor $B$, and this will be an equivalence. But the lax monoidality of the braiding is exactly what Equation \ref{ZXbraidingcompatible} entails, so by the Main Theorem from \cite{Wasserman2017a} we are done.
\end{proof}

The lax and oplax compatibility morphisms for the 2-fold monoidal structure (Equation \ref{ZXetazeta}) on $\dcentcat{A}$ give functors relating the convolution product and the enriched cartesian product of $\dcentcat{A}_s$-enriched categories. In proving this it will be convenient to introduce:

\begin{notation}
	For readability we will use the shorthand $\cat{K}_{ij}:=\cat{K}(k_i,k_j)$ with $i,j=0,1,2,3$, and the obvious version of this for $\cat{L}$. We will further sometimes suppress $\otimes_s$ from the notation and write $\cdot$ for $\cont$. 
\end{notation}

\begin{prop}\label{ZXcomparisonfunctors}
	The assignments
	\begin{align*}
	Z\colon (\cat{K}\cattens{s}\cat{L})\cattens{c}(\cat{K}'\cattens{s}\cat{L}')& \leftrightarrow (\cat{K}\cattens{c}\cat{K}')\cattens{s}(\cat{L}\cattens{c}\cat{L}'):H\\
	k\boxtimes l \boxtimes k'\boxtimes l' &\leftrightarrow k\boxtimes k' \boxtimes l \boxtimes l'\\
	(\cat{K}_{01}\cat{L}_{01})\cdot(\cat{K}_{01}'\cat{L}_{01}') &\xleftrightarrow[\eta]{\zeta} (\cat{K}_{01}\cdot\cat{K}_{01}')(\cat{L}_{01}\cdot\cat{L}_{01}'),
	\end{align*}
	are $\dcentcat{A}_s$-enriched functors. 
\end{prop}

\begin{proof}
	Composition in the category on the left hand side is given by:
	\begin{align*}
	((\cat{K}_{12}\cat{L}_{12})&\cdot(\cat{K}_{12}'\cat{L}_{12}'))((\cat{K}_{01}\cat{L}_{01})\cdot(\cat{K}_{01}'\cat{L}_{01}'))\xrightarrow{\eta} (\cat{K}_{12}\cat{L}_{12}\cat{K}_{01}\cat{L}_{01})\cdot(\cat{K}_{12}'\cat{L}_{12}'\cat{K}_{01}'\cat{L}_{01}')\\
	\xrightarrow{s\cdot s}& (\cat{K}_{12}\cat{K}_{01}\cat{L}_{12}\cat{L}_{01})\cdot(\cat{K}_{12}'\cat{K}_{01}'\cat{L}_{12}'\cat{L}_{01}')\xrightarrow{(\circ\circ) \cdot (\circ\circ)}(\cat{K}_{02}\cat{L}_{02})\cdot(\cat{K}_{02}'\cat{L}_{02}').
	\end{align*}
	On the right hand side, it is the composite:
	\begin{align*}
	(\cat{K}_{12}\cdot\cat{K}_{12}')&(\cat{L}_{12}\cdot\cat{L}_{12}')(\cat{K}_{01}\cdot\cat{K}_{01}')(\cat{L}_{01}\cdot\cat{L}_{01}')\\
	\xrightarrow{s}& (\cat{K}_{12}\cdot\cat{K}_{12}')(\cat{K}_{01}\cdot\cat{K}_{01}')(\cat{L}_{12}\cdot\cat{L}_{12}')(\cat{L}_{01}\cdot\cat{L}_{01}')\\
	\xrightarrow{\eta\eta}&((\cat{K}_{12}\cat{K}_{01})\cdot(\cat{K}_{12}'\cat{K}_{01}'))((\cat{L}_{12}\cat{L}_{01})\cdot(\cat{L}_{12}'\cat{L}_{01}'))\\
	\xrightarrow{(\circ\cdot\circ)(\circ\cdot\circ)}&(\cat{K}_{02}\cdot\cat{K}_{02}')(\cat{L}_{02}\cdot\cat{L}_{02}').
	\end{align*}
	Now, $\eta\otimes_s\eta$ gives a morphism from the first term in the second chain to the first term in the first chain, while $\eta$ gives a morphism between the last terms. Functoriality of $H$ is equivalent to the diagram formed in this way commuting. Comparing the penultimate terms in the sequences, we see that $\eta$ gives a map between these, and the square this forms with the composition morphisms and the final $\eta$ commutes by naturality of $\eta$. We are therefore left with showing that the rectangle formed by the first three terms in the sequences commutes. Schematically, this is the equation:
	$$
	\eta\circ(\eta\eta)\circ (\id s \id)= (s \cdot s)\circ\eta \circ (\eta \eta).
	$$
	We would like to use the fact that $\eta$ is compatible with the symmetry $s$ \cite[Definition \ref{DCbraideddef} and Theorem \ref{DCdc2foldmon}]{Wasserman2017a}, in the sense that: 
	$$
	\eta \circ s = (s\cdot s)\circ \eta
	$$
	However, for this we are using the symmetry on exactly the wrong factors. Fortunately, because $\eta$ respects the associators, we can replace 
	\begin{equation}
	\eta\circ(\eta\eta)= \eta \circ (\id  \eta) \circ (\id  \eta \id), \label{ZXetaass}
	\end{equation}
	to see
	\begin{align*}
	\eta\circ(\eta\eta)\circ s&=\eta \circ (\id  \eta) \circ (\id  \eta \id)\circ(\id s \id)=\eta \circ (\id  \eta) \circ (\id s\cdot s \id) \circ (\id  \eta \id)\\
	&=\eta \circ (\id s\cdot s)\circ (\id \eta)\circ (\id\eta\id)=(s\cdot s)\circ\eta \circ (\eta \eta),
	\end{align*}
	where the second equality is the compatibility of $\eta$ with the symmetry, and the final two equalities use the naturality of $\eta$, combined with Equation \eqref{ZXetaass}.
	Similarly, $\zeta$ is compatible with the symmetry and the associators, and so $Z$ is a functor.
\end{proof}

\begin{rmk}
	Because $\eta \circ \zeta=\id$, the functor $Z$ is in fact a right sided inverse to $H$.
\end{rmk}

We remind the reader that the forgetful functor $\mathbf{Forget}\colon\dcentcat{A}\rar \cat{A}$ takes $\otimes_c$ to $\otimes_{\cat{A}}$. This translates to the assignment $\cat{K}\mapsto\overline{ \cat{K}}$ taking the convolution product to the $\cat{A}$-enriched cartesian product.

\begin{lem}\label{ZXforgetconvprod}
	If $\cat{K}$ and $\cat{L}$ are $\dcentcat{A}_s$-enriched and tensored, then $\overline{\cat{K}\cattens{c}\cat{L}}= \overline{\cat{K}}\cattens{\cat{A}} \overline{\cat{L}}$.
\end{lem}

\begin{proof}
	To prove the claim, observe that the following diagram commutes:
	\begin{center}
		\begin{tikzcd}
		\dcentcat{A}\boxtimes\dcentcat{A} 	\arrow[r,"\otimes_c"] \arrow[d,"\mathbf{Forget}\boxtimes\mathbf{Forget}"']	&	\dcentcat{A} \arrow[d,"\mathbf{Forget}"]\\
		\cat{A}								\arrow[r,"\otimes_\cat{A}"] 										& 	\cat{A},
		\end{tikzcd},
	\end{center}
	because $\mathbf{Forget}$ is a strictly monoidal functor with respect to $\otimes_c$. This means =$\overline{\cat{K}\cattens{c}\cat{L}}$ and $ \overline{\cat{K}}\cattens{\cat{A}} \overline{\cat{L}}$ are the result of a change of basis along equal functors, and hence the same category.
\end{proof}

With respect to taking neutral subcategories, the convolution product behaves as follows:

\begin{lem}\label{ZXconvandneut}
	For $\cat{K},\cat{L}\in\dcentcat{A}\lincat$, we have
	$$
	\cat{K}_\cat{A}\cattens{c}\cat{L}_\cat{A}\subset\left(\cat{K}\cattens{c}\cat{L}\right)_\cat{A}.
	$$
\end{lem}
\begin{proof}
Recall that $\otimes_c$ restricts to the tensor product of $\cat{A}$ on the full monoidal subcategory $\cat{A}\subset\dcentcat{A}$. Observe that for any object $k\boxtimes l\in \cat{K}_\cat{A}\cattens{c}\cat{L}_\cat{A}$, the endomorphisms-object of $k\boxtimes l$ is the tensor product (in $\cat{A}\subset\dcentcat{A}$, see Lemma \ref{ZXneutralsubcatendochar}) of the endomorphisms of $k$ and $l$, hence an object in $\cat{A}$. Therefore $k\boxtimes l$ is an object of $(\cat{K}\cattens{c}\cat{L})_\cat{A}$. Similarly, the hom-object between any two pairs $k\boxtimes l$ and $k'\boxtimes l'$ is the tensor product of the respective hom-objects. By the definition of the neutral subcategory (Definition \ref{ZXneutralsubcat}), this is an object in $\cat{A}\subset\dcentcat{A}$, so lies in $(\cat{K}\cattens{c}\cat{L})_\cat{A}$.
\end{proof}

\subsubsection{$\dcentcat{A}$-crossed monoidal categories}
We can now define the notions of $\dcentcat{A}$-crossed tensor and braided categories. The basic idea is to define a notion of monoidal structure that factors through the crossed product $\cattens{c}$, instead of the enriched cartesian product $\cattens{s}$ one would normally use when defining monoidal objects in a category of enriched categories. A braiding for this crossed monoidal structure is a natural isomorphism between the monoidal structure and the monoidal structure precomposed with the appropriate swap map, just as one would do for ordinary braided monoidal categories. The difference here is that the swap map $B$ is not induced by a symmetric braiding, but rather by the non-degenerate braiding for $\otimes_c$ on $\dcentcat{A}$.
\begin{df}\label{ZXzacrossedtensordef}
	Let $\cat{K}$ be a $\dcentcat{A}_s$-enriched and tensored category. A \emph{$\dcentcat{A}$-crossed tensor structure on $\cat{K}$} is a functor:
	$$
	\otimes_\cat{K}\colon  \cat{K}\cattens{c}\cat{K}\rar \cat{K},
	$$
	together with a functor
	$$
	\mathbb{I}_\cat{K}\colon \cat{A}_\cat{Z}\rar \cat{K},
	$$
	and associators and unitors that satisfy the usual coherence conditions.
\end{df}

\begin{rmk}
It is worthwhile spelling out what kind of objects the unitors are in this context.  As the equivalence $\cat{A}_\cat{Z}\cattens{c}\cat{K}\cong \cat{K}$ is induced by the $\dcentcat{A}_s$-tensoring $-\cdot-$, the left unitor $\lambda$ is a natural isomorphism with components, for $a\in \cat{A}_\cat{Z}$ and $k\in \cat{K}$,
$$
\lambda_{a,k}:a\otimes_\cat{K} k \xrightarrow{\cong} (a\otimes_c \mathbb{I}_s)\cdot k,
$$
between the functors $\otimes_{\cat{K}}\circ(\mathbb{I} \cattens{c}\Id_\cat{K})$ and $\cat{A}_\cat{Z}\cattens{c}\cat{K}\cong \cat{K}$.
\end{rmk}

A little care has to be taken when considering functors and their natural transformations between such categories. Just like for ordinary monoidal categories, the appropriate notion of monoidal functor also entails the structure natural isomorphisms expressing the monoidality. We remark that, as being tensored is a property of an enriched category asking for the existence of an adjoint, any $\dcentcat{A}_s$-enriched functor will automatically preserve the tensoring up to a canonical isomorphism, by uniqueness of adjoints.

\begin{df}\label{ZXtensorfuncts}
	Let $(\cat{K},\otimes_\cat{K},\mathbb{I}_\cat{K})$ and $(\cat{L},\otimes_\cat{L},\mathbb{I}_\cat{L})$ be $\dcentcat{A}$-crossed tensor categories. A \emph{$\dcentcat{A}$-crossed tensor functor} between $\cat{K}$ and $\cat{L}$ is a triple $(F,\mu_0,\mu_1)$ consisting of a $\dcentcat{A}_s$-enriched functor $\cat{K}\rar \cat{L}$, and $\dcentcat{A}_s$-enriched natural isomorphisms
	$$
	\mu_0: F\circ\mathbb{I}_\cat{K}\Rightarrow \mathbb{I}_\cat{L},
	$$
	and 
	$$
	\mu_1: F(-\otimes_{\cat{K}}-)\Rightarrow F(-)\otimes_{\cat{L}}F(-).
	$$
	This data should satisfy the usual compatibility with the associators and the unitors. 
	
	A \emph{monoidal natural transformation} between two such functors $(F,\mu_0,\mu_1)$ and $(G,\nu_0,\nu_1)$ between $\cat{K}$ and $\cat{L}$ is a $\dcentcat{A}_s$-enriched natural transformation $\eta$ that makes the following diagrams commute:
	\begin{center}
		\begin{tikzcd}
		F(-\otimes_{\cat{K}}-) \arrow[r,"\mu_1",Rightarrow]\arrow[d,"\eta\circ(-\otimes -)",Rightarrow]	& F(-)\otimes_{\cat{L}}F(-)\arrow[d,"\eta\otimes \eta",Rightarrow]\\
		G(-\otimes_\cat{K} -)	\arrow[r,"\nu_{1}",Rightarrow]								& G(-)\otimes G(-),
		\end{tikzcd}
	\end{center}
	and
	\begin{center}
		\begin{tikzcd}
		F\circ \mathbb{I}_\cat{K}	\arrow[dd,"\eta\circ\mathbb{I}_\cat{K}",Rightarrow] \arrow[rd,"\mu_0",Rightarrow]	&						\\
																											&\mathbb{I}_\cat{L}\\
		G\circ \mathbb{I}_\cat{K}	\arrow[ru,Rightarrow,"\nu_0"]
		\end{tikzcd},
	\end{center}
	where the composite of $\eta$ with a functor is to be understood as whiskering.
\end{df}

\begin{rmk}
	The second diagram in our definition of a monoidal natural transformation is not present when dealing with ordinary monoidal categories.
\end{rmk}

\begin{ex}
	 The $\dcentcat{A}_s$-enriched and tensored category $\cat{A}_\cat{Z}$ is itself $\dcentcat{A}$-crossed tensor, and coherence ensures that $\mathbb{I}_\cat{K}$ is a $\dcentcat{A}_s$-crossed tensor functor for any $\dcentcat{A}$-crossed tensor category $\cat{K}$. Futhermore, if $(F,\mu_0,\mu_1)$ is a $\dcentcat{A}$-crossed monoidal functor from $\cat{K}$ to $\cat{L}$, the natural isomorphism $\mu_0$ is a monoidal natural transformation between $F\circ \mathbb{I}_\cat{K}$ and $\mathbb{I}_\cat{L}$.
\end{ex}

The bilax compatibility between $\otimes_c$ and $\otimes_s$ ensures that the enriched cartesian product of two $\dcentcat{A}$-crossed tensor categories is again a $\dcentcat{A}$-crossed tensor category.

\begin{prop}
	Let $\cat{K}$ and $\cat{L}$ be $\dcentcat{A}$-crossed tensor categories. Then $\cat{K}\cattens{s}\cat{L}$ is $\dcentcat{A}$-crossed tensor, with monoidal structure given by the componentwise tensor product.
\end{prop}

\begin{proof}
	The componentwise tensor product is given by the composite of $Z$ from Proposition \ref{ZXcomparisonfunctors} with the image under the change of basis along $\otimes_s$ of the $\dcentcat{A}_s\boxtimes \dcentcat{A}_s$-enriched functor
	$$
	\cat{K}\cattens{c}\cat{K}\boxtimes \cat{L}\cattens{c}\cat{L}\rar \cat{K}\boxtimes \cat{L}.
	$$
	As these are $\dcentcat{A}$-enriched functors, so is the componentwise tensor product. To establish associativity and unitality, we observe that $Z$ is compatible with the associators and unitors for $\dcentcat{A}$ and hence will preserve the componentwise associators.
\end{proof}

If $\cat{K}$ and $\cat{L}$ are $\dcentcat{A}$-crossed tensor, then the swap map $S$ for $\cattens{s}$ is a $\dcentcat{A}$-crossed tensor functor between $\cat{K}\cattens{s}\cat{L}$ and $\cat{L}\cattens{s}\cat{K}$, as we prove below. Together with the Proposition above, this implies that $\cattens{s}$ makes the subcategory of $\dcentcat{A}\lincat$ of $\dcentcat{A}$-crossed tensor categories, $\dcentcat{A}$-crossed tensor functors and monoidal natural transformations into a symmetric monoidal 2-category.

\begin{lem}\label{ZXsismonoidal}
	Let $\cat{K}$ and $\cat{L}$ be $\dcentcat{A}$-crossed tensor categories. Then the switch map $S\colon \cat{K}\cattens{s}\cat{L}\rar \cat{L}\cattens{s}\cat{K}$, that uses the symmetry in $\dcentcat{A}_s$ on hom-objects, is a $\dcentcat{A}$-tensor functor. In particular, the diagram
	\begin{center}
		\begin{tikzcd}
			\cat{K}\cattens{s}\cat{L}\cattens{c}\cat{K}\cattens{s}\cat{L}\arrow[r,"\otimes"] \arrow[d,"S\cattens{c}S"] &\cat{K}\cattens{s}\cat{L}\arrow[d,"S"]\\
			(\cat{L}\cattens{s}\cat{K})\cattens{c}(\cat{L}\cattens{s}\cat{K}) \arrow[r,"\otimes"]&\cat{L}\cattens{s}\cat{K}
		\end{tikzcd}
	\end{center}
	commutes strictly.
\end{lem}

\begin{proof}
	Unpacking the definition of the $\dcentcat{A}$-crossed monoidal structure on $\cat{K}\cattens{s}\cat{L}$, we get
	\begin{center}
		\begin{tikzcd}
			\cat{K}\cattens{s}\cat{L}\cattens{c}\cat{K}\cattens{s}\cat{L}\arrow[r,"Z"] \arrow[d,"S\cattens{c}S"] &\cat{K}\cattens{c}\cat{K}\cattens{s}\cat{L}\cattens{c}\cat{L} \arrow[r,"\otimes_{\cat{K}}\cattens{s}\otimes_\cat{L}"] \arrow[d,"S"]&\cat{K}\cattens{s}\cat{L}\arrow[d,"S"]\\
			(\cat{L}\cattens{s}\cat{K})\cattens{c}(\cat{L}\cattens{s}\cat{K}) \arrow[r,"Z"]&
			\cat{L}\cattens{c}\cat{L}\cattens{s}\cat{K}\cattens{c}\cat{K} \arrow[r,"\otimes_{\cat{K}}\cattens{s} \otimes_\cat{L}"]&\cat{L}\cattens{s}\cat{K}.
		\end{tikzcd}
	\end{center}
	The leftmost square commutes as a direct consequence of $\zeta$ being compatible with the symmetry, see Notation \ref{ZXnotation2fold}. The rightmost square commutes as the top route is change of basis along $\otimes_s$ for $\otimes_{\cat{K}}\boxtimes\otimes_\cat{L}$ composed with the switch functor, whereas the bottom is change of basis along $\otimes_s$ composed with the switch map on the same functor, and the symmetry is a natural isomorphism between these change of basis functors.
\end{proof}

Similarly to Proposition \ref{ZXmoncatpres}, we have:

\begin{lem}\label{ZXforgetmon}
	Let $\cat{K}$ be $\dcentcat{A}$-crossed tensor, then $ \overline{\cat{K}}$ is $\cat{A}$-tensor.
\end{lem}
\begin{proof}
	We have already established that $\overline{\cat{K}\cattens{c}\cat{K}}= \overline{\cat{K}}\cattens{\cat{A}} \overline{\cat{K}}$ in Lemma \ref{ZXforgetconvprod}, this means that the image under change of basis along the forgetful functor of the $\dcentcat{A}$-crossed monoidal structure is a functor $ \overline{\cat{K}}\cattens{\cat{A}} \overline{\cat{K}}\rar  \overline{\cat{K}}$. Furthermore, $\overline{\cat{A}_\cat{Z}}=\underline{\cat{A}}$, the category $\cat{A}$ enriched over itself, so the image of the unit for $\cat{K}$ is a functor $\underline{\cat{A}}\rar  \overline{\cat{K}}$, as required. By the 2-functoriality of the change of basis along the forgetful functor, the images of the unitor and associator will act as unitors and associators for $ \overline{\cat{K}}$.
\end{proof}

Recall that a monoidal category is called \emph{rigid} if every object has a dual, this notion translates verbosely to $\dcentcat{A}$-crossed tensor categories. In analogy with the notion of fusion category (\cite{Etingof2002}) in the setting of $\Vect$-enriched categories, we define:

\begin{df}\label{ZXZAfusiondef}
	If an FD $\dcentcat{A}_s$-category $\cat{K}$ is a $\dcentcat{A}$-crossed tensor category, rigid and is such that the unit $\mathbb{I}:\cat{A}_\cat{Z}\rar \cat{K}$ is a fully faithful functor, then $\cat{K}$ is called \emph{$\dcentcat{A}$-crossed fusion}.
\end{df}

\begin{rmk}
	In the definition of an ordinary fusion category, one usually asks that the unit object is a simple object. This implies the condition on the unit for a $\cat{Z}(\Vect)$-crossed fusion category. Conversely, as the unit object is simple in $\Vect$, full faithfulness implies that its image under the unit functor is a simple object.
\end{rmk}

\subsubsection{$\dcentcat{A}$-crossed braided categories}
\begin{df}\label{ZXcrossedbraidingdef}
	Let $\cat{K}$ be $\dcentcat{A}$-crossed tensor. Then a \emph{crossed braiding} for $\cat{K}$ is a natural isomorphism between $\otimes\colon  \cat{K}\cattens{c}\cat{K}\rar \cat{K}$ and
	$$
	\otimes\colon  \cat{K}\cattens{c}\cat{K}\xrightarrow{B}\cat{K}\cattens{c}\cat{K}\xrightarrow{\otimes} \cat{K},
	$$
	that satisfies the hexagon equations.
\end{df}

\begin{rmk}
	This definition implies that the unit functor $\mathbb{I}:\cat{A}_\cat{Z}\rar \cat{K}$ is a braided monoidal functor between crossed braided categories, where we equip $\cat{A}_\cat{Z}$ with its symmetry, by a modification of the usual argument that the hexagon equations imply compatibility between the braiding and the unitors, see for example \cite[Proposition 1]{Joyal1986}. In particular, in the case where $\cat{K}$ is $\dcentcat{A}$-crossed fusion as well as crossed braided, the unit functor is a braided monoidal embedding of $\cat{A}_\cat{Z}$ into $\cat{K}$.
\end{rmk}

The appropriate notion of a braided functor between such categories is the following:
\begin{df}\label{ZXbraidedfunct}
	Let $(F,\mu_0,\mu_1)$ be a $\dcentcat{A}$-crossed functor between crossed braided categories $(\cat{K},\otimes_K,\beta_K)$ and $(\cat{L},\otimes_L,\beta_L)$. Then $(F,\mu_0,\mu_1)$ is called \emph{braided} if the natural transformations
	$$
	\mu_1\circ F\circ\beta_\cat{K}\colon F(-\otimes_{\cat{K}}-)\Rightarrow F(-\otimes_{\cat{K}}-)B\Rightarrow (-\otimes_\cat{L}-)(F\cattens{c}F)B
	$$
	and
	$$
	\beta_\cat{L}\circ F\cattens{c}F\circ \mu_1\colon F(-\otimes_{\cat{K}}-)\Rightarrow (-\otimes_{\cat{L}}-)(F\cattens{c}F) \Rightarrow (-\otimes_\cat{L}-)B(F\cattens{c}F)
	$$
	agree. We remind the reader that as $B$ is induced by the braiding in $\dcentcat{A}$, and $F$ acts by morphisms in $\dcentcat{A}$ on hom-objects, we have that $(F\cattens{c}F)B=B(F\cattens{c}F)$, by naturality of the braiding.
\end{df}

\begin{prop}\label{ZXscrossedbraided}
	Let $\cat{K}$ and $\cat{L}$ be $\dcentcat{A}$-crossed braided categories. Then $\cat{K}\cattens{s}\cat{L}$ is $\dcentcat{A}$-crossed braided.
\end{prop}

\begin{proof}
	We will show that the componentwise braiding is compatible with the componentwise $\dcentcat{A}$-crossed monoidal structure. That is, 
	\begin{center}
		\begin{tikzcd}
			\cat{K}\cattens{s}\cat{L}\cattens{c}\cat{K}\cattens{s}\cat{L} \arrow[d,"B"] \arrow[r,"Z"] & \cat{K}\cattens{c}\cat{K}\cattens{s}\cat{L}\cattens{c}\cat{L} \arrow[r,"\otimes_{\cat{K}}\cattens{s}\otimes_\cat{L}"] \arrow[d,"B\cattens{s}B"] & \cat{K}\cattens{s}\cat{L} \arrow[d,"="]\\
			\cat{K}\cattens{s}\cat{L}\cattens{c}\cat{K}\cattens{s}\cat{L}  \arrow[r,"Z"] & \cat{K}\cattens{c}\cat{K}\cattens{s}\cat{L}\cattens{c}\cat{L} \arrow[r,"\otimes_{\cat{K}}\cattens{s}\otimes_\cat{L}"] & \cat{K}\cattens{s}\cat{L}			
		\end{tikzcd}
	\end{center}
	commutes up to componentwise braiding. The leftmost square commutes as a consequence of $\zeta$ satisfying a condition similar to Equation \eqref{ZXbraidingcompatible}, the rightmost square commutes up to the $\cattens{s}$ product of the braidings for $\cat{K}$ and $\cat{L}$.
\end{proof}

\begin{lem}\label{ZXsbraided}
	The switch functor $S$ from Lemma \ref{ZXsismonoidal} is braided monoidal.
\end{lem}

\begin{proof}
	We have to check that $S$ takes the componentwise braiding to the componentwise braiding. This is immediate from the symmetry and the braiding commuting with each other in $\dcentcat{A}$.
\end{proof}

Note that Proposition \ref{ZXscrossedbraided} and Lemma \ref{ZXsbraided} together imply that $\cattens{s}$ is a symmetric monoidal structure on the 2-category of $\dcentcat{A}$-crossed braided categories.

The neutral subcategory $\cat{K}_\cat{A}$ plays an interesting role. First of all we have:

\begin{prop}
	Let $\cat{K}$ be a $\dcentcat{A}$-crossed braided category. Then its neutral subcategory $\cat{K}_\cat{A}$ is a $\dcentcat{A}$-crossed braided subcategory.
\end{prop} 
\begin{proof}
	By Lemma \ref{ZXconvandneut} we have that
	$$
	\cat{K}_\cat{A}\cattens{c}\cat{K}_\cat{A}\subset (\cat{K}\cattens{c}\cat{K})_\cat{A}.
	$$
	This implies that the restriction of the crossed monoidal structure on $\cat{K}$ to $\cat{K}_\cat{A}$ factors through $\cat{K}_\cat{A}$, showing that it restricts to a monoidal stucture on $\cat{K}_\cat{A}$. The swap functor $B$ clearly restricts to a functor from $\cat{K}_\cat{A}\cattens{c}\cat{K}_\cat{A}$ to itself, and we see that the crossed braiding on $\cat{K}$ restricts to a crossed braiding on $\cat{K}_\cat{A}$.
\end{proof}

Recall that the swap functor $B$ is given by the braiding on hom-objects. On the subcategory $\cat{A}\subset\dcentcat{A}$, the braiding is just the symmetry in $\cat{A}$. This implies that the restriction of $B$ to $\cat{K}_\cat{A}\cattens{c}\cat{K}_\cat{A}$ is sent to the $\cat{A}$-swap functor (induced by the symmetry in $\cat{A}$) under change of basis along the forgetful functor $\dcentcat{A}\rar\cat{A}$. We therefore have the following partial analogue to Lemma \ref{ZXforgetmon}:

\begin{prop}
	Let $\cat{K}$ be a $\dcentcat{A}$-crossed tensor category. Then the image under change of basis along the forgetful functor $\dcentcat{A}\rar\cat{A}$ of its neutral subcategory $\overline{ \cat{K}_\cat{A}}$ is a braided $\cat{A}$-tensor category.
\end{prop}

\begin{rmk}
	Note that there is full analogue for Lemma \ref{ZXforgetmon}: $\dcentcat{A}$-crossed braided categories are not taken to braided $\cat{A}$-tensor categories by change of basis along the forgetful functor $\dcentcat{A}\rar \cat{A}$. This is because the forgetful functor is not a braided functor with respect to the braiding for $\otimes_c$, while restricted to the subcategory $\cat{A}\subset \dcentcat{A}$ it is. That no such analogue exists is a crucial point in future work \cite{Wassermanc}, where we show how to obtain $\dcentcat{A}$-crossed braided categories from braided categories containing $\cat{A}$, even in cases where the associated $\cat{A}$-enriched category fails to be braided.  
\end{rmk}

\subsubsection{A 2-category of $\dcentcat{A}$-crossed braided categories}
We can now formulate a 2-category of $\dcentcat{A}$-crossed braided categories:
\begin{df}\label{ZXzaxbtdef}
	The \emph{symmetric monoidal 2-category $\ZAXBT$ of $\dcentcat{A}$-crossed braided fusion categories} is the 2-category with
	\begin{itemize}
		\item objects $\dcentcat{A}$-crossed braided fusion categories,
		\item morphisms braided $\dcentcat{A}$-crossed tensor functors,
		\item 2-morphisms monoidal natural transformations,
		\item monoidal structure $\cattens{s}$, with swap map $S$.
	\end{itemize}
\end{df}

\section{$\dcentcat{A}$-Crossed Braided Categories and Tannaka Duality}\label{ZXdcxbsupertannasect}

This part of this paper is devoted to proving:

\begin{thm}\label{ZXzaxbtgxbteqv}
	Let $G$ (or $(G,\omega)$) be a finite (super) group. Then the functor $\overline{ \overline{(-)}}$ (see Definition \ref{ZXkbarbar} and Section \ref{ZXbarbarpsfunctor}) from $\ZAXBT$ (Definition \ref{ZXzaxbtdef}) to $\GXBT$ (or $\sGXBT$) (Definition \ref{ZXgcrossedbicat}) is a symmetric monoidal equivalence, with inverse given by $\mathbf{Fix}$ (see Definitions \ref{ZXfixobjdef}, \ref{ZXfix1morphs} and \ref{ZXfix2morphs}).
\end{thm}

We will start by introducing the relevant notions in Section \ref{ZXtannprelims}, in Section \ref{ZXfromzatog} we show how define the functor $\overline{ \overline{(-)}}$ and give an item by item proof that it lands in $G$-crossed braided categories. This the content of Theorem \ref{ZXbarbariscrossedbraided}. In Section \ref{ZXgxtozax} we show how to define $\mathbf{Fix}$ to produce from a $G$-crossed braided category a $\dcentcat{A}$-crossed tensor category.

\subsection{Preliminaries}\label{ZXtannprelims}

\subsubsection{The Drinfeld Centre as $G$-equivariant vector bundles}
By Tannaka duality \cite{Deligne1990, Deligne2002}, we can view any symmetric fusion category $\cat{A}$ as the representation category of some finite (super)-group. In this context, the appropriate notion of super-group is the following:
\begin{df}
	A \emph{super-group} $(G,\omega)$ is a group $G$ together with a choice of central element $\omega$ of order 2. A \emph{representation of a super-group} is a super-vector space $V\in \svec$ together with a homomorphism $G\rar \Aut (V)$, which takes $\omega$ to the grading involution of $V$. If $G$ is finite, these representations form a symmetric fusion category $\Rep(G,\omega)$, with symmetry inherited from $\svec$.
\end{df}

 The Drinfeld centre of the category of representations of a finite group $G$ is well-known \cite[Chapter 3.2]{Bakalov2001a} to be (braided monoidal) equivalent to the category of vector bundles over $G$, equivariant for the conjugation action of $G$ on itself. This result extends to the super-group case, the Drinfeld centre construction only uses the monoidal structure, not the braiding. Additionally, it was shown in \cite{Wasserman2017} that the symmetric tensor product on the Drinfeld centre agrees with the (graded) fibrewise tensor product.

\begin{df}
	A \emph{fibre functor} on a symmetric fusion category $\cat{A}$ is a braided monoidal functor
	$$
	\Phi\colon: \cat{A}\rar \svec.
	$$ 
	The case where the essential image of $\Phi$ is contained in $\Vect\subset\svec$ is referred to as \emph{Tannakian}, otherwise $\cat{A}$ is called \emph{super-Tannakian}.
\end{df}

Throughout the rest of this paper, we will fix a choice of fibre functor $\Phi$ on $\cat{A}$.

In what follows, we will make frequent use of the following facts. Viewing $\cat{A}$ as $\Rep(G)$ and $\dcentcat{A}$ as $\Vect_{G}[G]$, the forgetful functor $\dcentcat{A}\rar \cat{A}$ is given by direct sum over the fibres of the equivariant vector bundle over $G$. Applying the fibre functor $\Phi$ to a representation obtained in this way produces a vector space, which carries a $G$-grading by remembering over which elements the fibres sat. Additionally, this vector space carries a $G$-action which conjugates the grading.

\subsubsection{$G$-crossed braided categories}
\begin{df}[\cite{Turaev2010}]\label{ZXgcrossedbraided}
	A \emph{$G$-crossed braided fusion category} is a $\Vect$-enriched and tensored category $\cat{C}$, together with:
	\begin{enumerate}[(i)]
		\item\label{ZXdecompcross} for each $g\in G$ a $\Vect$-enriched and tensored semi-simple category $\cat{C}_g$ with finitely many simples, decomposing $\cat{C}$ as $\cat{C}=\oplus_{g\in G}\cat{C}_g$ (a \emph{$G$-graded linear category});
		\item\label{ZXgradedmon} a \emph{$G$-graded fusion structure}: a tensor structure $\otimes\colon \cat{C}\boxtimes\cat{C}\rar \cat{C}$, such that $\otimes\colon \cat{C}_g\boxtimes\cat{C}_h\rar \cat{C}_{gh}$, that is rigid with a simple unit;
		\item\label{ZXgaction} a homomorphism $G\rar \tn{Aut}(\cat{C})$. The image of $g\in G$ under this homomorphism will be denoted $(-)^g$. We require $(-)^g\colon \cat{C}_h\rar \cat{C}_{ghg^{-1}}$. (This is called a \emph{$G$-crossing}.)
		\item\label{ZXcrossedbraid} a \emph{crossed braiding}: for each $g\in G$ a natural isomorphism between $\otimes\colon \cat{C}_g\boxtimes \cat{C}\rar \cat{C}$ and
		$$
		\cat{C}_g\boxtimes \cat{C}\xrightarrow{S}\cat{C}\boxtimes\cat{C}_g \xrightarrow{(-)^g\boxtimes \tn{Id}}\cat{C}\boxtimes \cat{C}_g\xrightarrow{\otimes} \cat{C},
		$$
		satisfying the hexagon equations. Here $S$ is the swap functor for the Deligne product of linear categories.
	\end{enumerate}
\end{df}

\subsubsection{Super $(G,\omega)$-crossed braided categories}
Fix a supergroup $(G,\omega)$. We will now introduce the notion of a super $(G,\omega)$-crossed braided category, which the author believes to be new. Before we give the definition, we need to following.

\begin{df}\label{ZXgradinginvfunct}
	Let $\cat{C}$ be a $\svec$-enriched category. Then the \emph{grading involution functor} $\Pi$ on $\cat{C}$ is the autofunctor on $\cat{C}$ that acts as the identity on objects and even morphisms, and as $-\id$ on odd morphisms.
\end{df}

\begin{df}\label{ZXsgcrossedbraided}
	A \emph{super $(G,\omega)$-crossed braided category} is an $\svec$-enriched and tensored category $\cat{C}$, together with:
	\begin{enumerate}[(i)]
		\item\label{ZXsdecompcross} for each $g\in G$ an $\svec$-enriched and tensored category $\cat{C}_g$ that is semi-simple with finitely many simples, decomposing $\cat{C}$ as $\cat{C}=\oplus_{g\in G}\cat{C}_g$ (a \emph{$G$-graded super linear category});
		\item\label{ZXsgradedmon} a $G$-graded super fusion stucture: a super tensor structure $\otimes\colon \cat{C}\cattens{\svec}\cat{C}\rar \cat{C}$, such that $\otimes\colon \cat{C}_g\cattens{\svec}\cat{C}_h\rar \cat{C}_{gh}$, that is rigid and has a simple unit object;
		\item\label{ZXsgaction} a homomorphism $(G,\omega)\rar (\tn{Aut}(\cat{C}),\Pi)$ of pointed groups, where $\Pi$ denotes the grading involution functor. The image of $g \in G$ under this homomorphism will be denoted $(-)^g$. We require $(-)^g\colon \cat{C}_h\rar \cat{C}_{ghg^{-1}}$. (This is called a \emph{super $(G,\omega)$-crossing}.)
		\item\label{ZXscrossedbraid} a \emph{crossed braiding}, for each $g\in G$ a natural isomorphism between $\otimes\colon \cat{C}_g\cattens{\svec} \cat{C}\rar \cat{C}$ and
		$$
		\cat{C}_g\cattens{\svec} \cat{C}\xrightarrow{S}\cat{C}\cattens{\svec} \cat{C}_g \xrightarrow{(-)^g\boxtimes \tn{Id}}\cat{C}\cattens{\svec} \cat{C}_g\xrightarrow{\otimes} \cat{C},
		$$
		satisfying the hexagon equations. Here $S$ is the swap functor for $\cattens{\svec}$ which acts as the symmetry in $\svec$ on hom-objects.
	\end{enumerate}
\end{df}

For brevity, we will sometimes refer to a super $(G,\omega)$-crossed braided category as a super $G$-crossed braided category, leaving fixing a choice of $\omega$ implicit in the word super.

\subsubsection{The degreewise product of (super) $G$-crossed braided categories}
There is a natural notion of product of (super) $G$-crossed braided categories.
\begin{df}
	Let $\cat{C}$ and $\cat{D}$ be $G$-graded (super) linear categories, then the \emph{degreewise product of $\cat{C}$ and $\cat{D}$} is defined by:
	$$
	\cat{C}\cattens{G}\cat{D}=\bigoplus_{g\in G} \cat{C}_g\boxtimes\cat{D}_g,
	$$
	where in the super case we use $\cattens{\svec}$ instead of $\boxtimes$. Both these operations are a special case of Definition \ref{ZXaproddef}.
\end{df}

\begin{prop}
	The degreewise product $\cat{C}\cattens{G}\cat{D}$, of $G$-crossed braided categories $\cat{C}$ and $\cat{D}$, is $G$-crossed braided, for the componentwise tensor product, $G$-crossing and crossed braiding.
\end{prop}

\begin{proof}
	The category $\cat{C}\cattens{G}\cat{D}$ is $G$-graded by construction. The componentwise tensor product is given by:
	$$
	\cat{C}_g\boxtimes\cat{D}_g\boxtimes \cat{C}_h\boxtimes\cat{D}_h\xrightarrow{\Id\boxtimes S\boxtimes \Id}\cat{C}_g\boxtimes \cat{C}_h\boxtimes\cat{D}_g\boxtimes\cat{D}_h\xrightarrow{\otimes\boxtimes\otimes}\cat{C}_{gh}\boxtimes\cat{D}_{gh},
	$$
	where we use $\cattens{\svec}$ and its switch map instead of $\boxtimes$ in the super case. This clearly respects the $G$-grading. 
	
	Any pair of automorphisms of $\cat{C}$ and $\cat{D}$ that conjugate the grading will induce an automorphism of $\cat{C}\cattens{G}\cat{D}$ that conjugates the grading. In the super case, we observe that the $\svec$-product of the $\Z_2$-grading involution on $\svec$-enriched categories is the grading involution on the product, so the $G$-crossings of $\cat{C}$ and $\cat{D}$ will give an homomorphism of pointed groups as desired by \eqref{ZXsgaction} in Definition \ref{ZXsgcrossedbraided}.
	
	To see that the componentwise braiding gives a crossed braiding, observe that we need for $g\in G$ a natural isomorphism between the componentwise tensor product and
	$$
	(\cat{C}\cattens{G}\cat{D})_g \boxtimes (\cat{C}\cattens{G}\cat{D}) \xrightarrow{S}(\cat{C}\cattens{G}\cat{D})\boxtimes (\cat{C}\cattens{G}\cat{D})_g\xrightarrow{(-)_g\otimes \Id} (\cat{C}\cattens{G}\cat{D})\boxtimes (\cat{C}\cattens{G}\cat{D})_g\xrightarrow{\otimes}(\cat{C}\cattens{G}\cat{D}),
	$$
	where in the super case we replace $\boxtimes$ with $\cattens{\svec}$. That is, for each $h\in G$, the following diagram should commute up to the braiding:
	\begin{center}
		\begin{tikzcd}
			\cat{C}_g\boxtimes\cat{D}_g\boxtimes \cat{C}_h\boxtimes\cat{D}_h \arrow[r,"S"] \arrow[d,"\Id\boxtimes S \boxtimes \Id"] &  \cat{C}_h\boxtimes\cat{D}_h\boxtimes\cat{C}_g\boxtimes\cat{D}_g\arrow[d,"(-)_g\boxtimes (-)_g\boxtimes\Id\boxtimes\Id"] \\
			\cat{C}_g\boxtimes\cat{C}_h\boxtimes\cat{D}_g\boxtimes \cat{D}_h \arrow[d,"\otimes\boxtimes\otimes"] & \cat{C}_{ghg^{-1}}\boxtimes\cat{D}_{ghg^{-1}}\boxtimes\cat{C}_g\boxtimes\cat{D}_g \arrow[d,"\Id \boxtimes S \boxtimes \Id"]\\
			\cat{C}_{gh}\boxtimes\cat{D}_{gh} & \cat{C}_{ghg^{-1}}\boxtimes\cat{C}_g\boxtimes\cat{D}_{ghg^{-1}}\boxtimes\cat{D}_g\arrow[l,"\otimes \boxtimes \otimes"].
		\end{tikzcd}
	\end{center}
	Using that the switch map is natural, we can exchange the maps along the right hand side to get:
	\begin{center}
		\begin{tikzcd}
			\cat{C}_g\boxtimes\cat{D}_g\boxtimes \cat{C}_h\boxtimes\cat{D}_h \arrow[r,"S"] \arrow[d,"\Id\boxtimes S \boxtimes \Id"] &  \cat{C}_h\boxtimes\cat{D}_h\boxtimes\cat{C}_g\boxtimes\cat{D}_g\arrow[d,"\Id \boxtimes S \boxtimes \Id"] \\
			\cat{C}_g\boxtimes\cat{C}_h\boxtimes\cat{D}_g\boxtimes \cat{D}_h \arrow[d,"\otimes\boxtimes\otimes"]\arrow[r,"S\boxtimes S"] &    \cat{C}_h\boxtimes\cat{C}_g\boxtimes\cat{D}_h\boxtimes\cat{D}_g\arrow[d,"(-)_g\boxtimes (-)_g\boxtimes\Id\boxtimes\Id"]\\
			\cat{C}_{gh}\boxtimes\cat{D}_{gh} & \cat{C}_{ghg^{-1}}\boxtimes\cat{D}_{ghg^{-1}}\boxtimes\cat{C}_g\boxtimes\cat{D}_g\arrow[l,"\otimes \boxtimes \otimes"].
		\end{tikzcd}
	\end{center}
	The top square commutes strictly, and the bottom square indeed commutes up to the product of the braidings.
\end{proof}

The swap maps for $\boxtimes$ and $\cattens{\svec}$ induce a swap map for the degreewise product.

\subsubsection{A 2-category of (super) $G$-crossed braided fusion categories}
The (super) $G$-crossed categories fit into a symmetric monoidal 2-category:
\begin{df}\label{ZXgcrossedbicat}
	Let $G$ (or $(G,\omega)$) be a (super) group. Then the symmetric monoidal 2-category $\GXBT$ (or $\sGXBT$) has objects $G$- (or $(G,\omega)$-)crossed braided fusion categories. The 1-morphisms are (super) linear braided monoidal functors $F\colon \cat{C}\rar \cat{C}'$, satisfying $F(\cat{C}_g)\subset \cat{C}_g$ and $F\circ (-)^g=(F(-))^g$ for all $g\in G$. The 2-morphisms are monoidal natural transformations $\kappa$ satisfying $(\kappa_c)^g=\kappa_{c^g}$. The symmetric monoidal structure is given by the degreewise tensor product, with switch map given by the degreewise switch map of (super)-linear categories.
\end{df}

\begin{rmk}
	The definitions of (super) $G$-crossed braided category and functors between them used here are strict. One can also consider $G$-actions that are 2-functors from the 2-category with one object and no non-trivial 2-morphisms $G$ to the 2-category with one object $\Aut(\cat{C})$, and allow functors to preserve the $G$-action up to natural isomorphism.
\end{rmk}

\subsection{From $\dcentcat{A}$-crossed to (super) $G$-crossed}\label{ZXfromzatog}
In this section we will explain how to produce from a $\dcentcat{A}$-crossed braided fusion category a (super) $G$-crossed braided fusion category:

\begin{prop}\label{ZXbarbariscrossedbraided}
	Let $\cat{A}=\Rep(G)$ (or $\Rep(G,\omega)$). For any $\cat{K}$ be a $\dcentcat{A}$-crossed braided fusion category, the (super) linear category $ \overline{ \overline{\cat{K}}}$ obtained from $\cat{K}$ (see Definition \ref{ZXkbarbar} below) is (super) $G$-crossed braided fusion (see Definitions \ref{ZXgcrossedbraided} and \ref{ZXsgcrossedbraided}). 
\end{prop}

After this, in Section \ref{ZXbarbarpsfunctor}, we will show how to extend this to a 2-functor from $\ZAXBT$ to $\GXBT$ (or $\sGXBT$).

\subsubsection{The induced map for the fibre functor}
Given a $\dcentcat{A}_s$-enriched category $\cat{K}$, we can produce a $\Vect$-enriched and tensored category out of $\overline{ \cat{K}}$ (Definition \ref{ZXatensorforzaenr}) by changing basis along the fibre functor $\Phi$ for $\cat{A}$. The resulting category will usually not be idempotent complete, even when the original category was. We set:
\begin{df}\label{ZXkbarbar}
	Let $\cat{K}$ be a $\dcentcat{A}_s$-enriched and tensored category. Then \emph{the linearisation $ \overline{ \overline{\cat{K}}}$ of $\cat{K}$} is the Cauchy completion of $\Phi \overline{\cat{K}}$.
\end{df}

We observe that, as the fibre functor is unique up to monoidal natural isomorphism \cite{Deligne1990,Deligne2002}, the category $ \overline{ \overline{\cat{K}}}$ is unique up to equivalence. As the fibre functor is monoidal, $\Phi \overline{\cat{K}}$ is $\Vect$-enriched and tensored as a direct consequence of Proposition \ref{ZXbasechtens}.

\subsubsection{$G$-grading}
Since, on $\Vect_{G}[G]$, the forgetful functor followed by the fibre functor takes objects to $G$-graded (super) vector spaces, the morphisms of $\Phi \overline{\cat{K}}$ are $G$-graded (super) vector spaces. This will induce a grading on the idempotents:

\begin{lem}\label{ZXhomidemp}
	For every $\dcentcat{A}_s$-enriched and tensored category $\cat{K}$, every minimal idempotent of $\Phi\overline{\cat{K}}$ is homogeneous for the $G$-grading on the hom-objects.
\end{lem}

\begin{proof}
	Let $k$ be an object of $\Phi\overline{\cat{K}}$. Composition of endomorphisms of $k$ factors through the image of the symmetric tensor product of $\cat{K}(k,k)$ with itself. This image is the fibrewise (super) tensor product (\cite{Wasserman2017}). Observe that an idempotent is necessarily even. In the super-case, the fibrewise super tensor product reduces to the fibrewise tensor product for even objects. Decomposing an even endomorphism $\psi$ into homogeneous components $\psi_g$, the condition for $\psi$ to be an idempotent becomes:
	$$
	\psi \circ \psi = \sum_{g\in G} \psi_g \circ \psi_g = \sum_{g\in G} \psi_g=\psi,
	$$
	which is a condition for each $\psi_g$ separately. So $\psi$ is idempotent if and only if all its homogeneous components are. In particular, any minimal idempotent is homogeneous.
\end{proof}

This means that there is a function from the simple objects of the category $ \overline{ \overline{\cat{K}}}$ to $G$. We would like to extend this to a direct sum decomposition of our category, so we need to establish:
\begin{lem}\label{ZXkbbhomobjlem}
	Let $k$ and $k'$ be simple objects of $\overline{ \overline{\cat{K}}}$ of degrees $g$ and $g'$, respectively. Then $\overline{ \overline{\cat{K}}}(k,k')$ is zero unless $g=g'$. 
	
	Furthermore, assume that $k=f_k\in \Phi\overline{ \cat{K}}(\bar{k},\bar{k})$ and $k'=f_k'\in \Phi\overline{ \cat{K}}(\bar{k}',\bar{k}')$ for objects $\bar{k},\bar{k}'\in\Phi\overline{ \cat{K}}$ and idempotents $f_k,f_k'$. Then, denoting by $\Phi\overline{ \cat{K}}(\bar{k},\bar{k}')_{g,p}$ taking the even $(p=0)$ or odd $(p=1)$ part of the summand $\Phi\overline{ \cat{K}}(\bar{k},\bar{k}')_{g}$, any morphism of parity $p$ between $k$ and $k'$ arises from composing a morphism in $\Phi\overline{ \cat{K}}(\bar{k},\bar{k}')_{\omega^p g,p}$ with the idempotents.
\end{lem}

\begin{proof}
	In an idempotent completion, the hom-object between two objects is computed by pre- and post-composing with the idempotents. Composition factors over the fibrewise (super) tensor product, so morphisms between $k$ and $k'$ are in the image of the composition map out of:
	$$
	\Phi\overline{ \cat{K}}(\bar{k}',\bar{k}')^{g'}_0\otimes_f^{\omega}\Phi\overline{ \cat{K}}(\bar{k},\bar{k}')\otimes_f^\omega\Phi\overline{ \cat{K}}(\bar{k},\bar{k}')_0^g,
	$$
	where $\otimes_f^\omega=\otimes_f$ in the non-super case, and $\Phi\overline{ \cat{K}}(\bar{k}',\bar{k}')^{g'}_0$ denotes the bundle supported by $\{g'\}$ with fibre $\Phi\overline{ \cat{K}}(\bar{k}',\bar{k}')_{g',0}$.  We have used similar notation for the rightmost factor. Computing the rightmost product, we see this is the bundle with fibres:
	$$
	(\Phi\overline{ \cat{K}}(\bar{k},\bar{k}')\otimes_f^\omega\Phi\overline{ \cat{K}}(\bar{k},\bar{k}')_0^g)_{h,p}=\Phi\overline{ \cat{K}}(\bar{k},\bar{k}')_{h,p}\otimes(\Phi\overline{ \cat{K}}(\bar{k},\bar{k}')^g_0)_{\omega^ph}.
	$$
	For this to be non-zero, we need $h=\omega^pg$, proving the ``furthermore'' part of the Lemma. Taking the fibrewise (super) tensor product of this with $\Phi\overline{ \cat{K}}(\bar{k}',\bar{k}')^{g'}$, we get:
	\begin{align*}
	(\Phi\overline{ \cat{K}}(\bar{k}',\bar{k}')^{g'}_0\otimes_f^{\omega}\Phi\overline{ \cat{K}}(\bar{k},\bar{k}')&\otimes_f^\omega\Phi\overline{ \cat{K}}(\bar{k},\bar{k}')_0^g)_{h',p'}=\\
	&(\Phi\overline{ \cat{K}}(\bar{k}',\bar{k}')^{g'}_0)_{\omega^{p'}h'}\otimes(\Phi\overline{ \cat{K}}(\bar{k},\bar{k}')\otimes_f^\omega\Phi\overline{ \cat{K}}(\bar{k},\bar{k}')_0^g)_{h',p'}.
	\end{align*}
	We immediately see that for this to be non-zero requires $\omega^{p'}h'=g'$, and from the above computation $\omega^{p'}h'=g$, so we need $g=g'$ for this to be non-zero.
\end{proof}

Combining this with Lemma \ref{ZXhomidemp}, we get:

\begin{cor}\label{ZXforgetdecomposes}
	The $\Vect$-enriched and tensored category $ \overline{ \overline{\cat{K}}}$ decomposes a direct sum
	$$
	\overline{ \overline{\cat{K}}}=\bigoplus_{g\in G}  \overline{ \overline{\cat{K}}}_g
	$$
	of (super-)linear categories, where $\overline{ \overline{\cat{K}}}_g$ is the additive subcategory of $\overline{ \overline{\cat{K}}}$ generated by the minimal idempotents of degree $g$.
\end{cor}

We remind the reader that this corresponds to items \eqref{ZXdecompcross} from Definition \ref{ZXgcrossedbraided} and \eqref{ZXsdecompcross} from Definition \ref{ZXsgcrossedbraided}.

\subsubsection{FD to Fusion}
An integral part of Proposition \ref{ZXbarbariscrossedbraided} is the assertion that FD $\dcentcat{A}_s$-linear categories are sent to idempotent complete linear semi-simple categories.

\begin{prop}
	Let $\cat{K}$ be a FD $\dcentcat{A}_s$-linear category. Then $\overline{ \overline{\cat{K}}}$ is an idempotent complete linear semi-simple category.
\end{prop}

\begin{proof}
	The $\Vect$-enriched and tensored category $\overline{ \overline{\cat{K}}}$ is idempotent complete and additive by definition, so we only need to establish semi-simplicity. Given a monic morphism $i:a\hookrightarrow b$ in $\overline{ \overline{\cat{K}}}$, we want to show this monic splits. To do this, it is enough to provide an idempotent $f$ on $b$ such that $fi$ is an isomorphism between $a$ and the subobject associated to $f$. Tracing through the definition of $\overline{ \overline{\cat{K}}}$, we see that $i$, up to composition with some idempotents, comes from a monic in $\cat{K}$. This monic corresponds in turn to a monic in $\DE(\cat{K})$, where it splits by assumption. Tracing this splitting through the constructions gives the splitting in $\overline{ \overline{\cat{K}}}$.
\end{proof}

\subsubsection{Graded monoidal structure}

\begin{lem}\label{ZXforgettensor}
	$ \overline{ \overline{\cat{K}}}$ is (super)-fusion.
\end{lem}

\begin{proof}
	This is a direct consequence of Proposition \ref{ZXmoncatpres} and Lemma \ref{ZXforgetmon}.
\end{proof}

This (super-)tensor structure is graded in the sense that it satisfies item \eqref{ZXgradedmon} from Definition \ref{ZXgcrossedbraided} (or item \eqref{ZXsgradedmon} from Definition \ref{ZXsgcrossedbraided}):

\begin{lem}
	The (super)-tensor structure from Lemma \ref{ZXforgettensor} maps
	$$
	\overline{ \overline{\cat{K}}}_g \boxtimes  \overline{ \overline{\cat{K}}}_h \rar  \overline{ \overline{\cat{K}}}_{gh},
	$$
	with respect to the decomposition from Corollary \ref{ZXforgetdecomposes}, and where we replace $\boxtimes$ with $\cattens{\svec}$ in the super case.
\end{lem}

\begin{proof}
	The grading of the (super-)tensor product of two homogeneous objects $k\in \overline{ \overline{\cat{K}}}_g$ and $k'\in \overline{ \overline{\cat{K}}}_h$ obtained by taking the tensor product of the idempotents in $\Phi \overline{K}$. This tensor product, in turn, factors over the convolution tensor product of $G$-graded (super) vector spaces, and this sends the $g$-graded and the $h$-graded part to the $gh$-graded part, so the (super-)tensor product of the idempotents will be homogeneous of degree $gh$.
\end{proof}

\subsubsection{$G$-action}
The $G$-graded vector spaces obtained by applying the forgetful and fibre functors to the objects of $\dcentcat{A}$ carry an action of the (super)-group, that we will denote by $g\cdot$. This $G$-action conjugates the grading. This action of the (super-)group translates to an action on the idempotent completion:
\begin{lem}
	Let $g\in G$, then the assignment
	\begin{align*}
	(-)^g\colon &	 \overline{ \overline{\cat{K}}}\rar \overline{ \overline{\cat{K}}}\\
	&k \mapsto g\cdot k\\
	& \overline{ \overline{\cat{K}}}(k,k') \rar g\cdot \overline{ \overline{\cat{K}}}(g\cdot k,g\cdot k')
	\end{align*}
	defines an autofunctor of $ \overline{ \overline{\cat{K}}}$. 
\end{lem}

\begin{proof}
	This assignment is clearly strictly invertible, with inverse given by $(-)^{g^{-1}}$, so we have to prove that it defines a functor. It is enough to show that the assignment
	\begin{align*}
	(-)^g\colon  &\Phi \overline{\cat{K}} \rar \Phi \overline{\cat{K}}\\
	&k\mapsto k\\
	&\Phi \overline{\cat{K}}(k,k')\rar g\cdot\Phi \overline{\cat{K}}(k,k')
	\end{align*}
	is a functor, this will descend to the idempotent completion $ \overline{ \overline{\cat{K}}}$ as prescribed. As the identity morphisms are given by equivariant maps from $\mathbb{I}_s=\C\times G$ to the hom-objects, $(-)^g$ preserves identities. Recall that composition maps out of the fibrewise tensor product, and is a morphism in $\dcentcat{A}$. Any morphism in $\dcentcat{A}$ is a morphisms intertwining the $G$-action on the vector bundles over $G$, so we are trying to show that the fibrewise tensor product has the property that
	$$
	g\cdot (V\otimes_f W)= (g\cdot V)\otimes_f (g\cdot W),
	$$
	but this true by definition, see \cite[Definitions \ref{STfibrewisetensorproductdef} and \ref{STfibrewisesupertensordef}]{Wasserman2017}.
\end{proof}

\begin{lem}
	The assignment $g\mapsto (-)^g$	defines a homomorphism $G\rar\tn{Aut}( \overline{ \overline{\cat{K}}})$, or $(G,\omega)\rar(\tn{Aut}( \overline{ \overline{\cat{K}}}),\Pi)$ in the super case.
\end{lem}

\begin{proof}
	In the non-super case, there is nothing to prove. In the super-case we observe that the $\Z_2$-grading on the $G$-equivariant vector bundles over $G$ is exactly determined by whether $\omega$ acts by $1$ or $-1$. So, $(-)^{\omega}$ will act by $-1$ exactly on the odd morphisms, i.e. as $\Pi$ (Definition \ref{ZXgradinginvfunct}). 
\end{proof}

We observe that this action of $g\in G$ on $ \overline{ \overline{\cat{K}}}$ takes $ \overline{ \overline{\cat{K}}}_h$ to $ \overline{ \overline{\cat{K}}}_{ghg^{-1}}$. This means that this action satisfies item \eqref{ZXgaction} from Definition \ref{ZXgcrossedbraided} (or from Definition \ref{ZXsgcrossedbraided} in the super case).

\subsubsection{$G$-crossed braiding}
To prove Theorem \ref{ZXbarbariscrossedbraided}, we still need to show that $ \overline{ \overline{\cat{K}}}$ satisfies \eqref{ZXcrossedbraid} from Definition \ref{ZXgcrossedbraided} (or \eqref{ZXscrossedbraid} from Definition \ref{ZXsgcrossedbraided}). The first step for this is:

\begin{lem}\label{ZXimageofB}
	The image of the functor $B$ (see Definition \ref{ZXbraidingfunctor}) under the change of basis along the forgetful functor followed by $\Phi$ and then idempotent completion is given by:
	\begin{equation}\label{ZXimageofbraiding}
	\overline{ \overline{\cat{K}}}_g\boxtimes \overline{ \overline{\cat{K}}}\xrightarrow{\tn{Switch}}\overline{ \overline{\cat{K}}}\boxtimes \overline{ \overline{\cat{K}}}_g \xrightarrow{(-)^g\boxtimes \Id} \overline{ \overline{\cat{K}}}\boxtimes \overline{ \overline{\cat{K}}}_g,
	\end{equation}
	where in the super case, we use $\cattens{\svec}$ instead of $\boxtimes$, and the switch map uses the symmetry in super vector spaces.
\end{lem}

\begin{proof}
	The image of the functor $B$ (see Definition \ref{ZXbraidingfunctor}) under the change of basis along the forgetful functor followed by $\Phi$ is given by
	\begin{align*}
	\Phi  \overline{B}\colon & \Phi\overline{K}\boxtimes \Phi\overline{K}\rar\Phi\overline{K}\boxtimes \Phi\overline{K}\\
	& k\boxtimes k' \mapsto k'\boxtimes k,
	\end{align*}
	and the image of the braiding on hom-objects. In the model of $\dcentcat{A}$ as $\Vect_G[G]$, this braiding is given fibrewise by:
	$$
	V_g\otimes W_h \rar (g\cdot W_h)\otimes V_g.
	$$
	Without loss of generality, let $k_1,k_1'$ and $k_2,k_2'$ be simple objects of $\overline{ \overline{\cat{K}}}$ of degrees $g$ and $g'$, respectively. Then, by Lemma \ref{ZXkbbhomobjlem}, morphisms $f\colon k_1\rar k_1'$ of parity $p$ and $f'\colon k_2\rar k_2'$ of parity $p'$ come from fibres over $\omega^pg$ or $\omega^{p'} g'$, respectively. This means that the image of $B$ will take $f\otimes f'$ to $(\omega^p g\cdot f') \otimes f$, which, remembering that $\omega$ acts non-trivially only if $p'=1$, we can rewrite as $(-1)^{pp'} g\cdot f'\otimes f$. But this is exactly what the composite from Equation \eqref{ZXimageofbraiding} does.
\end{proof}

Recall that the braiding for a $\dcentcat{A}$-crossed braided category $\cat{K}$ is a natural isomorphism between $\otimes_\cat{K}\colon \cat{K}\cattens{c}\cat{K} \rar \cat{K}$ and the composite $\otimes_{\cat{K}}\circ B$. By Proposition \ref{ZXchangebasisispseudofunct}, this descends to a natural isomorphism between the images of these functors, so Lemma \ref{ZXimageofB} has the following consequence:

\begin{cor}
	The braiding for a $\dcentcat{A}$-crossed braided fusion category $\cat{K}$ descends to a natural isomorphism between $\overline{ \overline{\otimes}}\colon\overline{ \overline{\cat{K}}}_g\boxtimes \overline{ \overline{\cat{K}}}\rar\overline{ \overline{\cat{K}}}$ and
	$$
	\overline{ \overline{\cat{K}}}_g\boxtimes \overline{ \overline{\cat{K}}}\xrightarrow{\tn{Switch}}\overline{ \overline{\cat{K}}}\boxtimes \overline{ \overline{\cat{K}}}_g \xrightarrow{(-)^g\boxtimes \Id} \overline{ \overline{\cat{K}}}\boxtimes \overline{ \overline{\cat{K}}}_g \xrightarrow{\overline{ \overline{\otimes}}}\overline{ \overline{\cat{K}}}.
	$$
\end{cor}

Because this braiding satisfies coherence, so will its image. This shows that $\overline{ \overline{\cat{K}}}$ satisfies item \eqref{ZXcrossedbraid} from Definition \ref{ZXgcrossedbraided} (or item \eqref{ZXscrossedbraid} from Definition \ref{ZXsgcrossedbraided}).

This completes the proof of Proposition \ref{ZXbarbariscrossedbraided}.

\subsubsection{The assignment $\overline{ \overline{(-)}}$ is a 2-functor to $\GXBT$ (or $\sGXBT$)}\label{ZXbarbarpsfunctor}
The aim of this subsection is to show that $\overline{ \overline{(-)}}$ is a 2-functor from $\ZAXBT$ to $\GXBT$ (or $\sGXBT$ in the super group case). We first show that it takes functors of $\dcentcat{A}$-crossed braided categories to functors on $G$-crossed braided fusion categories.

\begin{prop}
	Let $F\colon\cat{K}\rar \cat{K}'$ be a 1-morphism in $\ZAXBT$ (Definition \ref{ZXzaxbtdef}). Then $\overline{ \overline{F}}$ is a 1-morphism in $\GXBT$ (or $\sGXBT$ in the super group case).
\end{prop}

\begin{proof}
	The fact that $\overline{ \overline{F}}$ is a (super) linear braided monoidal functor is immediate from Proposition \ref{ZXsymlaxinducessymspseudo} and the fact that $F$ is braided monoidal. We still have to show that $\overline{ \overline{F}}$ respects the direct sum decomposition of $\overline{ \overline{\cat{K}}}$ and the $G$-action. For the former, observe that $F$ acts by morphisms of $\dcentcat{A}$ on the hom-objects. Viewing $\dcentcat{A}$ as $\Vect_{G}[G]$, these morphisms are maps of vector bundles over $G$, so will descend to $G$-grading preserving morphisms, and will in particular send idempotents of degree $g$ to idempotents of the same degree. Similarly, on hom-objects $F$ will act by $G$-equivariant maps, this implies that $\overline{ \overline{F}}$ will be $G$-equivariant.
\end{proof}

We also need that 2-morphisms in $\ZAXBT$ are sent to 2-morphisms in $\GXBT$ (or $\sGXBT$).

\begin{prop}
	Let $\kappa$ be a 2-morphism in $\ZAXBT$ between $F,G\colon\cat{K}\rar \cat{K}'$. Then $\overline{ \overline{\kappa}}$ is a 2-morphism in $\GXBT$ (or $\sGXBT$ in the super group case.)
\end{prop}

\begin{proof}
	It is clear that $\overline{ \overline{\kappa}}$ will be monoidal. To see that it satisfies $(\overline{ \overline{\kappa}}_k)^g=\overline{ \overline{\kappa}}_{k^g}$ for each $k\in\overline{ \overline{\cat{K}}}$ and $g\in G$, recall that a component $\kappa_c$ of the enriched natural transformation is a morphism
	$$
	\kappa_c\colon \mathbb{I}_s\rar \cat{K}'(F(c),G(c)).
	$$
	In $\Vect_{G}[G]$, we have $\mathbb{I}_s=\C\times G$, so $\kappa_c$ is constant on each conjugacy class of $G$. Now, for $k\in \overline{ \overline{\cat{K}}}$ homogeneous of degree $h$, the object $k^g$ is homogeneous of degree $ghg{-1}$, that is, it comes from an idempotent of conjugate degree on the same object. But as $\overline{ \overline{\kappa}}_k$ is defined by precomposing the image of $\kappa$ with these idempotents under forget and fibre, this means that $\overline{ \overline{\kappa}}$ satisfies the condition $(\overline{ \overline{\kappa}}_k)^g=\overline{ \overline{\kappa}}_{k^g}$.
\end{proof}

\subsubsection{Degreewise tensor product}
We now show that the assignment $\cat{K}\mapsto \overline{ \overline{\cat{K}}}$ takes the product $\cattens{s}$ (Definition \ref{ZXcattenssdef}) to the degreewise tensor product $\cattens{G}$.
\begin{prop}
	The 2-functor $\overline{ \overline{(-)}}$ takes the enriched Cartesian product of $\dcentcat{A}$-crossed braided fusion categories to the degreewise product of $G$-crossed braided fusion categories.	
\end{prop}

This will be a consequence of:

\begin{lem}
	Let $\cat{K}$ and $\cat{L}$ be $\dcentcat{A}$-crossed fusion categories. Then
	$$
	\overline{\overline{\cat{K}\cattens{s}\cat{L}}}=\overline{\overline{\cat{K}}}\cattens{G} \overline{ \overline{\cat{L}}}.
	$$
\end{lem}

\begin{proof}
	From Lemma \ref{ZXforgetonsymprod}, we know how to compare the forgetful image of the $\dcentcat{A}_s$-enriched Cartesian product with the $\cat{A}$-enriched Cartesian product. Applying the fibre functor and idempotent completing gives functors:
	$$
	\overline{ \overline{H}}\colon\overline{\overline{\cat{K}\cattens{s}\cat{L}}}\leftrightarrow\overline{\overline{\cat{K}}}\boxtimes \overline{ \overline{\cat{L}}}\colon\overline{ \overline{Z}},
	$$
	with $\overline{ \overline{Z}}\overline{ \overline{H}}=\Id$. We claim that the image of $\overline{ \overline{H}}$ is $\overline{\overline{\cat{K}}}\cattens{G} \overline{ \overline{\cat{L}}}$, the result will then follow. To see this, observe that, when viewing $\overline{\dcentcat{A}}$ as $G$-graded (super) vector spaces over $G$, $\overline{\eta}$ is the morphism that takes the degreewise product and includes it into the convolution product. This means that $\overline{H}$ will descend to $\overline{ \overline{H}}$ as the functor that takes homogeneous idempotents to their degreewise product, which is what we wanted to show.
\end{proof}

This completes the proof of Theorem \ref{ZXzaxbtgxbteqv}.

\subsection{From $G$-crossed braided fusion categories to $\dcentcat{A}$-crossed braided fusion categories}\label{ZXgxtozax}
In this section, we will give a construction that produces $\dcentcat{A}$-crossed braided categories from $G$-crossed braided categories, and then extend this to a symmetric monoidal bifunctor $\mathbf{Fix}$. This uses a variation of the $G$-fixed category construction (see for example \cite{Turaev2010a}). 

\subsubsection{The $G$-fixed category}
\begin{df}\label{ZXfixobjdef}
	Let $\cat{C}$ be a (super) $G$-crossed (or $(G,\omega)$-crossed) braided fusion category. Then the \emph{$G$-fixed category $\cat{C}^G$} is the $\dcentcat{A}_s$-enriched and tensored category with objects pairs $(c,\{u_g\}_{g\in G})$, where $c$ is an object of $\cat{C}$, and the $u_g\colon(c)^g\xrightarrow{\cong} c$ are (even) isomorphisms such that:
	\begin{center}
		\begin{tikzcd}
			(c)^{gh} \arrow[r,"\cong"]\arrow[d,"u_{gh}"] & ((c)^h)^g \arrow[d,"(u_h)^g"]\\
			c & (c)^{g} \arrow[l,"u_g"]
		\end{tikzcd}
	\end{center}
	commutes for all $g,h\in G$. The hom-objects $\cat{C}^G((c,u),(c',u'))\in \dcentcat{\cat{A}}$ are given by 
	$$
	\cat{C}^G((c,u),(c',u'))=(\cat{C}(c,c'),\mathfrak{b}),
	$$
	where we equip $\cat{C}(c,c')$ with the $G$-action:
	$$
	g\cdot\colon \cat{C}(c,c')\xrightarrow{(-)^g}\cat{C}((c)^g,(c')^g)\xrightarrow{(u_{g})^*\circ(u_{g^{-1}})_*} \cat{C}(c,c'),
	$$
	and $\mathfrak{b}$ is the half-braiding defined by, for every $a=(V,\rho)\in\Rep (G)$:
	$$
	\mathfrak{b}\colon\cat{C}(c,c') a = \bigoplus_{g\in G} \cat{C}(c_g,c'_g) V \xrightarrow{\mathbf{Switch}} \bigoplus_{g\in G} V  \cat{C}(c_g,c'_g)\xrightarrow{\oplus \rho(g)\otimes \id }\bigoplus_{g\in G}V \cat{C}(c_g,c'_g),
	$$
	where we have used the direct sum decomposition of $\cat{C}$ and used subscript $g$ to denote the homogeneous components, and the switch map is the switch map of (super) vector spaces. Examining the definition (see for example \cite[Definition \cite{STequivvbbraiding}]{Wasserman2017}) of the half-braiding in $\Vect_{G}[G]$, we see that this half-braiding corresponds to taking $\cat{C}^G((c,u),(c',u'))$ to be the equivariant vector bundle with fibre over $g$ given by:
	$$
	\cat{C}^G((c,u),(c',u'))_g=\cat{C}(c_g,c'_g)_0\oplus\cat{C}(c_{\omega g},c'_{\omega g})_1,
	$$
	where the subscripts $0$ and $1$ denote taking the even and odd summands respectively.	Composition is given by the composition of $\cat{C}$.
\end{df}

\begin{rmk}
	The reader might observe that this is a variation of the homotopy fixed point construction for the $G$-action.
\end{rmk}

\begin{lem}
	The $G$-fixed category is indeed a $\dcentcat{A}_s$-enriched and tensored category.
\end{lem}

\begin{proof}
	Using the main results from \cite{Wasserman2017}, we can view $\dcentcat{A}_s$ as the category $\Vect_{G}[G]$ of $G$-equivariant vector bundles over $G$, equipped with the (super) fibrewise tensor product, that we will denote by $\otimes_f$ in both cases.
	
	We need to show that the composition of $\cat{C}$ defines a morphism:
	$$
	\cat{C}^G((c',u'),(c'',u''))\otimes_f	\cat{C}^G((c,u),(c',u'))\rar \cat{C}^G((c,u),(c'',u'')),
	$$
	that is $G$-equivariant, factors over the (super) fibrewise tensor product, and is compatible with the specified braiding. For the $G$-equivariance, we simply observe that $u_g'\circ u_{g^{-1}}'=\id$. To see the composition factors over the fibrewise (super) tensor product, observe that the direct sum decomposition of $\cat{C}$ implies that any two morphisms $f\colon c_g\rar c_g'$ and $f'\colon c'_h\rar c_h''$ will compose to $0$ unless $g=h$. For the even part of the hom-objects, this immediately implies that the composition factors through the fibrewise tensor product. To examine what happens for the odd parts of the hom-objects, we will start by assuming that one of the morphisms is odd, say the one between $c'$ and $c''$. In this case the fibrewise super tensor product computes as (see \cite[Definition \ref{STfibrewisesupertensordef}]{Wasserman2017}), using the notation from the proof of Lemma \ref{ZXkbbhomobjlem}:
	$$
	\cat{C}(c'_{\omega g},c''_{\omega g})^g_1\otimes_f \cat{C}(c_{g'},c'_{g'})^{g'}_{0}=\begin{cases}
	\left(\cat{C}(c'_{\omega g},c''_{\omega g})_1\cat{C}(c_{\omega g},c'_{\omega g})_{0}\right)^{g}& \tn{ for }g'=\omega g\\ 0 &\tn{ otherwise.}
	\end{cases} 
	$$
	We see that this corresponds again to morphisms of different degrees composing to zero. The case where the other morphism is odd is similar. If both are odd, we compute:
	$$
	\cat{C}(c'_{\omega g},c''_{\omega g})^g_1\otimes_f \cat{C}(c_{\omega g'},c'_{\omega g'})^{ g'}_{1}=\begin{cases}
	\left(\cat{C}(c'_{\omega g},c''_{\omega g})_1\cat{C}(c_{\omega g},c'_{\omega g})_{1}\right)^{\omega g}& \tn{ for }g'= g\\ 0 &\tn{ otherwise,}
	\end{cases} 
	$$
	from which we again see that the composition factors over the fibrewise super tensor product.
	The specified braiding is exactly the one $\Vect_{G}[G]$, so this same observation implies that the composition morphism commutes with the braiding.
\end{proof}

This construction takes $G$-crossed braided fusion categories to $\dcentcat{A}$-crossed braided fusion categories.
\begin{prop}
	If $\cat{A}=\Rep(G)$ (or $\cat{A}=\Rep(G,\omega)$), then $\cat{C}^G$ is a $\dcentcat{A}$-crossed braided fusion category, with $\dcentcat{A}$-crossed tensor structure given by:
	\begin{align*}
	\otimes\colon 	&	\cat{C}^G\cattens{c}\cat{C}^G \rar \cat{C}^G\\
	&	(c,u) \boxtimes (c',u') \mapsto (cc',u\otimes u'),
	\end{align*}
	and on morphisms by the monoidal structure in $\cat{C}$. The $\dcentcat{A}$-crossed braiding is the natural transformation with the same components as the crossed braiding on $\cat{C}$.
\end{prop}
\begin{proof}
	The first step is to show the monoidal structure on morphisms really factors over the convolution tensor product. We observe that, as the monoidal structure on $\cat{C}$ is graded, we have:
	$$ 
	(c_1c_2)_g=\bigoplus_{g_1g_2=g} c_{1,g_1}c_{2,g_2}.
	$$
	This gives a decomposition of the hom-object
	$$
	\cat{C}(c_1c_2,c_1c_1')=\bigoplus_{g\in G}\bigoplus_{g_1g_2=g}\cat{C}(c_{1,g_1}c_{2,g_2},c_{1,g_1}'c_{2,g_2}').
	$$
	From this, we see that the monoidal structure in $\cat{C}$ will indeed factor over the convolution product. To see that the crossed braiding induces a $\dcentcat{A}$-crossed braiding, we observe that the half-braiding on $\cat{C}^G(c,c')$ restricts to the $G$-action on the summands.
\end{proof}

\subsubsection{The 2-functor $\mathbf{Fix}$}
We now want to extend the $G$-fixed category construction to functors and natural transformations of (super) $G$-crossed braided categories.

\begin{defprp}\label{ZXfix1morphs}
	Let $F\colon \cat{C}\rar \cat{C}'$ be a 1-morphism in $\GXBT$ (or $\sGXBT$). Then we define the \emph{associated $G$-fixed functor $\mathbf{Fix}(F)$} as 
	$$
	F\colon  (c,u)\mapsto (F(c),F(u))
	$$ 
	on objects and by 
	$$
	F_{c,c}\colon \cat{C}(c,c')\rar \cat{C}'(Fc,Fc')
	$$
	on hom-objects. This is a 1-morphism in $\ZAXBT$.
\end{defprp}
\begin{proof}
	We need to show that this prescription indeed defines a $\dcentcat{A}_s$-enriched functor that is braided monoidal. On objects, there is nothing to show. On hom-objects, we need to show that $F$ acts by morphisms in $\dcentcat{A}$, so is compatible with the prescribed half-braiding, this follows from the $G$-equivariance of $F$. The fact that $\mathbf{Fix}(F)$ is braided monoidal is immediate from the definition of the $\dcentcat{A}$-crossed braided monoidal structures on $\cat{C}^G$ and $\cat{C}^{\prime,G}$.
\end{proof}

To extend $\mathbf{Fix}$ to 2-morphisms, we define:
\begin{df}\label{ZXfix2morphs}
	Let $\kappa$ be a 2-morphism in $\GXBT$ (or $\sGXBT$) between $F,F'\colon \cat{C}\rar \cat{C}'$. Then $\mathbf{Fix}(\kappa)$ is the $\dcentcat{A}$-enriched natural transformation with components:
	$$
	\mathbf{Fix}(\kappa)_{(c,u)}\colon  \mathbb{I}_s\rar \cat{C}^{\prime}(F(c),F'(c)),
	$$
	given fibrewise by $\kappa_{c_g}\colon \C\times\{g\} \rar \cat{C}^{\prime}(F(c_g),F'(c_g))$. 
\end{df}

\subsection{Equivalence between $\ZAXBT$ and $\GXBT$ (or $\sGXBT$)}\label{ZXequivgxzx}
We will now show that the 2-functors $\overline{ \overline{(-)}}$ and $\mathbf{Fix}$ are mutually inverse, this will complete the proof of Theorem \ref{ZXzaxbtgxbteqv}:

\begin{prop}\label{ZXfixbarbarmutinv}
	The 2-functors $\overline{ \overline{(-)}}$ and $\mathbf{Fix}$ are mutually inverse.
\end{prop}

As a first step, we will show that $\mathbf{Fix}(\overline{ \overline{\cat{K}}})$ is equivalent to $\overline{ \overline{\cat{K}}}$. To do this, we will need the following two technical lemmas:

\begin{lem}\label{ZXessimidemcomp}
	Let $F\colon \cat{C}\rar \cat{D}$ be a fully faithful functor on an idempotent complete category $\cat{C}$. Then the essential image of $F$ is idempotent complete.
\end{lem}

\begin{proof}
	Suppose that $f \in \End_{\cat{D}}(F(c))$ is an idempotent. By full faithfulness of $F$, this $f$ is the image of a unique $g\in \End_\cat{C}(c)$, which is necessarily idempotent. By idempotent completeness of $\cat{C}$, there exists an object $b\hookrightarrow c$ corresponding to $g$, which is mapped to a subobject $F(b)\hookrightarrow F(c)$ corresponding to $f$ under the functor $F$.
\end{proof}

\begin{lem}\label{ZXnonzerotrick}
	Suppose that for each object $c$ in an abelian category $\cat{C}$ we have a natural assignment $c \mapsto (i(c)\colon c \rar B(c))$, and that for every non-zero $c$ the map $i(c)$ is non-zero. Then $i(c)$ is monic for all $c$.
\end{lem}

\begin{proof}
	Suppose that $i(c)$ had some kernel $\mathfrak{k}\colon k\hookrightarrow c$. Then applying our natural assignment to $k$ gives the commutative diagram:
	\begin{center}
		\begin{tikzcd}
			k \arrow[r,"i(k)"] \arrow[d,hook,"\mathfrak{k}"] & B(k)\arrow[d,"B(\mathfrak{k})"]\\
			c \arrow[r,"i(c)"] &	B(c)	.		
		\end{tikzcd}
	\end{center}
	The composite along the bottom is zero, as $\mathfrak{k}$ is the kernel of $i(c)$. By naturality, $B(\mathfrak{k})$ is the kernel of $B(i(c))$, and therefore monic. So the bottom composite being zero implies that $i(k)$ is zero, implying that the kernel is trivial.
\end{proof}

We are now in a position to prove that $\cat{K}$ and $\mathbf{Fix}(\overline{ \overline{\cat{K}}})$ are equivalent.

\begin{lem}\label{ZXHexists}
	For each $\dcentcat{A}$-crossed braided fusion category $\cat{K}$ there is an equivalence of $\dcentcat{A}$-crossed tensor categories
	$$
	\cat{H}_\cat{K}\colon \cat{K}\rar\mathbf{Fix}(\overline{ \overline{\cat{K}}}),
	$$
	given by taking $k\in \cat{K}$ to $(\id_k, \{\id_k\}_{g\in G})$, and on hom-objects by the isomorphism
	$$
	\cat{K}(k,k')\cong (\Phi(\overline{ \cat{K}}), \rho, \mathfrak{b}),
	$$
	where $\rho$ denotes the $G$-action on $\Phi \cat{K}$ coming from the $G$-action on $\cat{K}(k,k')$, and $\mathfrak{b}$ its half-braiding.
\end{lem}

\begin{proof}
	This functor is fully faithful by definition, so we only need to establish essential surjectivity. That is, for every $(f,u)$, we need to give an isomorphism to an object $(\id_k,\{\id_k\}_{g\in G})$. When $f$ is zero, this is trivial, so assume $f$ is non-zero. As the essential image of a fully faithful functor on an idempotent complete category is idempotent complete (Lemma \ref{ZXessimidemcomp}), it suffices to find a monic morphism
	$$
	(f,u)\rar_{\mathbb{I}_s} (\id_k,\{\id_k\}_{g\in G}),
	$$
	this will then correspond to a subobject of $(\id_k,\{\id_k\}_{g\in G})$, which is necessarily in the essential image. If $(f,u)$ has as underlying idempotent $f\in \Phi\overline{ \cat{K}}(k',k')$, we will produce a morphism to
	$$
	\C[G]^*\cdot (\id_{k'},\{\id_{k'}\}_{g\in G})=(\id_k,\{\id_k\}_{g\in G}),
	$$
	where $k= \C[G]^* k'$ and we equip $\C[G]$ with the left action of $G$. To produce this morphism, observe that $f$ defines a morphism in $\overline{ \overline{\cat{K}}}(f,\id_{k'})$, and therefore gives rise to a morphism in $\mathbf{Fix}(\overline{ \overline{\cat{K}}})((f,u),(\id_{k'},\{\id_{k'}\}_{g\in G}))$. The image under the $G$-action for $g\in G$ (see Definition \ref{ZXfixobjdef}) of $f$ is:
	$$
	f \mathop{\longmapsto}\limits^{(-)^g} f^g \mathop{\longmapsto}\limits^{-\circ(u_g)^{-1}} f^gu_g^{-1}.
	$$
	By adjunction in $\dcentcat{A}_s$, a morphism degree $\mathbf{I}_s$ to $(\id_k,\{\id_k\}_{g\in G})$ is the same as a fibrewise map:
	$$
	\tilde{f}\colon  \C[G]\otimes_s \mathbb{I}_s \rar \mathbf{Fix}(\overline{ \overline{\cat{K}}})((f,u),(\id_{k'},\{\id_{k'}\}_{g\in G}))
	$$
	that is equivariant for the $G$ action. The $G$-equivariant vector bundle $\C[G]\otimes_s \mathbb{I}_s$ is the bundle $\C[G]\times G$, where $G$ acts on the fibres by left multiplication. We define $\tilde{f}$ by $f|_{\{g\}\times G} = f^gu_g^{-1}$. To show that the morphism
	$$
	\mathbb{I}_s\rar \mathbf{Fix}(\overline{ \overline{\cat{K}}})((f,u),(\id_{k},\{\id_{k}\}_{g\in G}))
	$$
	is obtained in this way is monic, by Lemma \ref{ZXnonzerotrick} it suffices to show that $\tilde{f}$ is non-zero. Restricting $\tilde{f}$ to $\{e\}\times G$ gives $f$, which is assumed to be non-zero. Therefore, we have produced a monic morphism from $(f,u)$ to $(\id_k,\{\id_k\}_{g\in G})$, which by Lemma \ref{ZXessimidemcomp} corresponds to an subobject of the form $(\id_l,\{\id_l\}_{g\in G})$ for some $l\in\cat{K}$.
	
	It is clear from the definition of $\cat{H}_\cat{K}$ that it will be a functor of $\dcentcat{A}$-crossed braided categories.
\end{proof}

As a second step, we will show that the $\cat{H}_\cat{K}$ are natural in the sense that:

\begin{lem}\label{ZXHisnatural}
	Let $F\colon \cat{K}\rar \cat{K}'$ be a 1-morphism in $\ZAXBT$. Then $F$ and the image of $F$ under $\mathbf{Fix}\circ\overline{ \overline{(-)}}$ fit into a commutative diagram:
	\begin{center}
		\begin{tikzcd}
			\cat{K} \arrow[r,"\cat{H}_\cat{K}"] \arrow[d,"F"] & \mathbf{Fix}(\overline{ \overline{\cat{K}}}) \arrow[d,"\mathbf{Fix}(\overline{ \overline{F}})"]\\
			\cat{K}' \arrow[r,"\cat{H}_{\cat{K}'}"]& \mathbf{Fix}(\overline{ \overline{\cat{K'}}}).
		\end{tikzcd}
	\end{center}
\end{lem}
\begin{proof}
	Let $k\in\cat{K}$. Under the top composite in the diagram, this object is sent to 
	$$
	(F(\id_k),F(\{\id_k\}_{g\in G})=(\id_{F(k)},\{\id_{F(k)}\}_{g\in G}),
	$$
	and the bottom composite is the same. On morphisms, it is similarly clear that the diagram commutes on hom-objects.
\end{proof}

The two Lemmas \ref{ZXHexists} and \ref{ZXHisnatural} together imply that $\mathbf{Fix}\circ\overline{ \overline{(-)}}$ is isomorphic to the identity on $\ZAXBT$. For the composite the other way around, we first prove:

\begin{lem}\label{ZXPexists}
	Let $\cat{C}$ be a (super)-G-crossed braided fusion category. Then the categories $\overline{ \overline{\mathbf{Fix}(\cat{C})}}$ and $\cat{C}$ are equivalent.
\end{lem}

\begin{proof}
	We will define a dominant fully faithful functor
	$$
	\hat{\cat{P}}_\cat{C}\colon \Phi \overline{ \mathbf{Fix}(\cat{C})} \rar \cat{C},
	$$
	after idempotent completion this will descent to an equivalence of categories. On objects, this functor is given by $(c,u)\mapsto c$, on morphisms we use the isomorphism
	$$
	\Phi\overline{ \mathbf{Fix}}((c,u),(c',u'))\cong \cat{C}(c,c'),
	$$
	as $\Phi\overline{(-)}$ simply forgets the $G$-action and half-braiding. This functor is clearly fully faithful. To see that it is dominant, observe that for any object $c\in\cat{C}$, the object $\oplus_{g\in G} c^g$ is a fixed point for the $G$-action, and therefore $(\oplus_{g\in G} c^g, \{\id\}_{g\in G})$ defines an object in $\mathbf{Fix}(\cat{C})$, which our functor will send to $\oplus_{g\in G} c^g$. This object has $c$ as a summand, so we see our functor is indeed dominant.
\end{proof}

\begin{lem}\label{ZXPisnatural}
	Denote the equivalence from Lemma \ref{ZXPexists} by
	$$
	\cat{P}_\cat{C}\colon  \overline{ \overline{\mathbf{Fix}(\cat{C})}}\rar\cat{C}.
	$$
	Then, for any 1-morphism $F$ in $\GXBT$ (or $\sGXBT$), the diagram
	\begin{center}
		\begin{tikzcd}
			\overline{ \overline{\mathbf{Fix}(\cat{C})}}\arrow[r,"\cat{P}_\cat{C}"]\arrow[d,"\overline{ \overline{\mathbf{Fix}(F)}}"] &\cat{C} \arrow[d,"F"]\\
			\overline{ \overline{\mathbf{Fix}(\cat{C}')}}\arrow[r,"\cat{P}_{\cat{C}'}"]&\cat{C}'
		\end{tikzcd}
	\end{center}
	commmutes.
\end{lem}

\begin{proof}
	It suffices to show that the diagram
	\begin{center}
		\begin{tikzcd}
			\Phi \overline{\mathbf{Fix}(\cat{C})}\arrow[r,"\hat{\cat{P}}_\cat{C}"]\arrow[d,"\Phi\overline{\mathbf{Fix}(F)}"'] &\cat{C} \arrow[d,"F"]\\
			\Phi \overline{\mathbf{Fix}(\cat{C}')}\arrow[r,"\hat{\cat{P}}_{\cat{C}'}"]&\cat{C}'
		\end{tikzcd}
	\end{center}
	commutes, as it will descent to the desired diagram after idempotent completion. On an object $(c,u)\in \Phi \overline{\mathbf{Fix}(\cat{C})}$, the bottom route becomes 
	$$
	(c,u)\mapsto (Fc,Fu) \mapsto Fc,
	$$
	which agrees with the top route. A similar diagram chase shows that this diagram commutes on hom-objects.
\end{proof}

The Lemmas \ref{ZXPexists} and \ref{ZXPisnatural} together show that the composite $\overline{ \overline{\mathbf{Fix}(-)}}$ is naturally isomorphic to the identity on $\GXBT$ (or $\sGXBT$).

This finishes the proof of Proposition \ref{ZXfixbarbarmutinv}, and with that, the proof of Theorem \ref{ZXzaxbtgxbteqv}.

\appendix
\section{Appendix}
\setcounter{thm}{0}
\renewcommand{\thethm}{\thesection.\arabic{thm}}
\setcounter{equation}{0}
\renewcommand{\theequation}{\thesection.\arabic{equation}}

\subsection{Enriched and tensored categories}
Fix a spherical symmetric fusion category $\cat{A}$ with unit object $\mathbb{I}$ throughout. We assume the reader is familiar with the basic definition of a category enriched in $\cat{A}$. This section will deal with categories that are not only enriched, but also tensored over $\cat{A}$.

\subsubsection{Basics}
\begin{notation}\label{ZXhomsetnotation}
	The hom-objects in an $\cat{A}$-enriched category $\cat{C}$ between $c,c'\in\cat{C}$ will be denoted by $\underline{\cat{C}}(c,c')$.
	We will write $f\colon c\rar_a c'$ for $f\colon a\rar\underline{\cat{C}}(c,c')$. If $a=\mathbb{I}$, we will omit it from the notation. Furthermore, we will write $\cat{C}(c,c')$ for $\cat{A}(\mathbb{I},\underline{\cat{C}}(c,c'))$.
\end{notation}

We remind the reader of the following definition.

\begin{df}\label{ZXnattrafo}
	Let $F,G\colon \cat{C}\rar\cat{D}$ be functors of $\cat{A}$-enriched categories. An \emph{enriched natural transformation} from $F$ to $G$ is for each object $c\in \cat{C}$ a morphism $\eta_{c}\colon F(c)\rar_{\mathbb{I}}G(c)$, that makes the following diagram commute for any $f\colon c\rar_a c'\in\cat{C}$:
	\begin{center}
		\begin{tikzcd}
			F(c)\arrow[r,"\eta_c"]\arrow[d,"F(f)", "a"' very near end]& G(c)\arrow[d,"G(f)", "a"' very near end]\\
			F(c')\arrow[r,"\eta_{c'}"]&G(c').
		\end{tikzcd}
	\end{center}
\end{df}

\begin{df}\label{ZXtensoreddef}
	Let $\cat{C}$ be a category enriched in $\cat{A}$. Then $\cat{C}$ is called \emph{tensored over $\cat{A}$} if there exists, for every $c,c'\in \cat{C}$ and $a\in\cat{A}$ an object $a\cdot c$ together with a functorial isomorphism
	\begin{equation}\label{ZXcopowernattrafo}
	\cat{A} \left( a, \underline{\cat{C}}(c,c')\right)\cong \cat{C}( a\cdot c,c').
	\end{equation}
\end{df}

Definition \ref{ZXtensoreddef} allows us to write, denoting by $\cat{O}(\cat{A})$ a set of representatives for the isomorphism classes of simple objects in $\cat{A}$:
\begin{equation}\label{ZXehom}
\underline{\cat{C}}(c,c')\cong \bigoplus_{a\in\cat{O}(\cat{A})} \cat{C}(ac,c')a.
\end{equation}
This means we can view $f\colon c\rar_a c'$ as a morphism $\tilde{f}\colon ac\rar c'$, and the composite of $f\colon c\rar_a c'$ and $g\colon c'\rar_{a'} c''$ is given by 
$$
\tilde{g}\circ\left(\id_{a'}\cdot\tilde{f}\right)\colon a'ac\rar c''.
$$

\begin{df}\label{ZXmates}
	The images of morphisms under the isomorphism \eqref{ZXcopowernattrafo} are called \emph{mates}. For $f\colon c\rar_a c'$ we will write $\bar{f}\colon ac\rar c'$ for its mate, and the mate of $g\colon ac\rar c'$ will be denoted by $\underline{g}\colon c\rar_a c'$.
\end{df}

\begin{rmk}
	If $\cat{C}$ is enriched over $\cat{A}$, its category of enriched endofunctors $\End(\cat{C})$ is a tensor category, with the monoidal structure coming from composition. The assignment $a\mapsto a\cdot -$ extends to a functor $\cat{A}\rar \End(\cat{C})$. This functor is in fact monoidal, cf. Lemma \ref{ZXmonoidality}.
\end{rmk}

Categories enriched and tensored over $\cat{A}$ form a 2-category, where we do have to take care functors between them are compatible with the tensor structure:

\begin{df}\label{ZXalincatdef}
	The \emph{2-category of $\cat{A}$-categories} $\lcat{A}$ is the 2-category where 
	\begin{itemize}
		\item objects are categories enriched in and tensored over $\cat{A}$,
		\item morphisms $\cat{A}$-enriched functors $F\colon \cat{C}\rar \cat{C}'$ equipped with a natural isomorphism
		$$
		F(a\cdot c) \xrightarrow[\cong]{\mu_{a,c}} a \cdot F(c),
		$$
		monoidal in $a$, such that the diagrams
		\begin{center}
			\begin{tikzcd}
				\cat{C}(a\cdot c,c') \arrow[rr,"\cong"] \arrow[d,"F"] 	&	& \cat{A}(a,\cat{C}(c,c'))	\arrow[d,"F"]\\
				\cat{C}'(F(ac),F(c'))	&  \cat{C}'(a\cdot F(c),F(c')) \arrow[l,"\mu"]	& \cat{A}(a,\cat{C}'(c,c')) \arrow[l,"\cong"],
			\end{tikzcd}
		\end{center}
		commute for all $a\in\cat{A}$ and $c,c'\in \cat{C}$,
		\item and 2-morphisms enriched natural transformations $\eta\colon F\Rightarrow G$ that make the diagrams 
		\begin{center}
			\begin{tikzcd}
				F(ac) \arrow[r,"\eta_{ac}"] \arrow[d,"\cong"] 	&	G(ac)	\arrow[d,"\cong"]\\
				a\cdot F(c)	\arrow[r,"\id_{a}\cdot \eta_{c}"]	& a\cdot G(c),
			\end{tikzcd}
		\end{center}
		commute for all $a\in\cat{A}$ and $c\in\cat{C}$.
	\end{itemize}  
\end{df}

\begin{rmk}
	Definition \ref{ZXalincatdef} is the most restrictive of the possible choices for a definition of $\lcat{A}$. We could also have allowed that there is an auto-equivalence of $\cat{A}$ associated to every morphism between $\cat{A}$-enriched and tensored categories, and that each 2-morphism comes with a symmetric monoidal transformation between these auto-equivalences. However, Definition \ref{ZXalincatdef} corresponds to the kind of enriched and tensored categories one obtains from module categories over $\cat{A}$.
\end{rmk}

\begin{df}\label{ZXaselfenriched}
	The \emph{internal hom} between two objects $a,a'\in\cat{A}$ is the representing object $\underline{\cat{A}}(a,a')$ for the functor $a''\mapsto\cat{A}(a''a,a')$. These hom objects make $\cat{A}$ into a category enriched and tensored over itself, i.e. a closed monoidal category.
\end{df}

\begin{lem}\label{ZXunitinternalhom}
	There is a canonical isomorphism
	\begin{equation}
	\underline{\cat{A}}(\mathbb{I},a)\cong a.
	\end{equation}
	for all objects $a\in\cat{A}$.
\end{lem}
\begin{proof}
	Consider that
	$$
	\cat{A}(a',\underline{\cat{A}}(\mathbb{I},a))\cong \cat{A}(a',a),
	$$
	so $\cat{A}(\mathbb{I},a)$ and $a$ are canonically isomorphic under the Yoneda embedding.
\end{proof}

The tensor structure of $\cat{C}$ over $\cat{A}$ induces an enriched natural isomorphism $\eta$ with components
\begin{equation}\label{ZXtensenriched}
\eta_{a,c,c'}\colon \underline{\cat{A}}(a,\underline{\cat{C}}(c,c'))\rar \underline{\cat{C}}(ac,c').
\end{equation}
To see this, observe that, given the natural transformation from (\ref{ZXcopowernattrafo}), the definition of the enriched hom for $\cat{A}$ gives
\begin{align}
\cat{A}(a',\underline{\cat{A}}(a,\underline{\cat{C}}(c,c')))&\cong\cat{A}(a'a,\underline{\cat{C}}(c,c'))\\
&\cong\cat{A}(a',\underline{\cat{C}}(ac,c'))
\end{align}
where the second line uses the isomorphism from (\ref{ZXcopowernattrafo}). The preimage of this isomorphism under the Yoneda embedding is the desired isomorphism.

\begin{lem}\label{ZXmonoidality}
	Suppose $\cat{C}$ is tensored over $\cat{A}$. Then the functor $\cat{A}\rar \End(\cat{C})$ taking $a$ to $a\cdot-$ is a tensor functor.
\end{lem}
\begin{proof}
	We only prove that there exist isomorphisms $a\cdot(a'\cdot c)\cong aa'\cdot c$ and omit checking the triangle and pentagon equations. We observe that
	\begin{align}
	\begin{split}
	\cat{C}(a\cdot(a'\cdot c),c')&\cong \cat{A}(a,\underline{\cat{C}}(a'c,c'))\\
	& \cong \cat{A}(a,\underline{\cat{A}}(a',\underline{\cat{C}}(c,c')))\\
	&\cong \cat{A}(aa',\underline{\cat{C}}(c,c'))\\
	&\cong \cat{C}(aa'c,c'),
	\end{split}
	\end{align}
	for all $c'\in\cat{C}$.
\end{proof}

\begin{df}\label{ZXevaluation}
	We let
	\begin{equation}
	\tn{ev}\colon \underline{\cat{C}}(c,c')\cdot c \rar c'.
	\end{equation} 
	be the unit of the adjunction between $\underline{\cat{C}}(c,-)\colon \cat{C}\rar\cat{A}$ and $-\cdot c\colon \cat{A}\rar\cat{C}$ from \eqref{ZXcopowernattrafo}.
\end{df}

\begin{rmk}
	The evaluation morphism allows us to rewrite the defining diagram for an enriched natural transformation (Definition \ref{ZXnattrafo}) as follows:
	\begin{center}
		\begin{tikzcd}
			
			\underline{\cat{C}}(c,c')\cdot F(c)\arrow[r,"\eta_c"]\arrow[d,"F\cdot \tn{id}"]	&	\underline{\cat{C}}(c,c')\cdot  G(c)\arrow[d,"G\cdot \tn{id}"]	\\
			\underline{\cat{C}}(Fc,Fc')\cdot F(c)\arrow[d,"\tn{ev}"]	&	\underline{\cat{C}}(c,c')\cdot  G(c)\arrow[d,"\tn{ev}"]	\\
			F(c')\arrow[r,"\eta_c"]												&	G(c').
		\end{tikzcd}
	\end{center}
\end{rmk}

\begin{lem}\label{ZXaintohom}
	There exists a canonical isomorphism
	$$
	a  \underline{\cat{C}}(c,c')\rar \underline{\cat{C}}(c,ac').
	$$
\end{lem}
\begin{proof}
	We construct an isomorphism between the images of $a\underline{\cat{C}}(c,c')$ and $\underline{\cat{C}}(c,ac')$ under the Yoneda embedding:
	\begin{align}\begin{split}
	\cat{A}(a',a\underline{\cat{C}}(c,c'))&\cong \cat{A}(a^*a',\underline{\cat{C}}(c,c'))\\
	&\cong \cat{C}(a^*a'c,c')\\
	&\cong \cat{C}(a'c,ac')\\
	&\cong \cat{A}(a',\underline{\cat{C}}(c,ac')).
	\end{split}
	\end{align}
\end{proof}

\subsubsection{Abelian structure}
To understand the analogues of the notions in linear categories that we want to find here, it is helpful to first revisit the linear case. One can phrase the definition of a linear category as follows: a linear category $\cat{C}$ is a category enriched over $\Vect$ (the category of finite dimensional vector spaces), which is abelian (with respect to the abelian group structure of the vector spaces). As $\cat{C}$ is abelian, it then further makes sense to ask for semi-simplicity, ie. that all short exact sequences split. 

In the setting a category $\cat{C}$ enriched over a symmetric fusion category $\cat{A}$  we have to take a bit more care, as our morphisms are no longer elements of vectors spaces, and therefore a priori do not form an abelian group. If we assume our category is not only enriched, but also tensored, we can view a morphism $f:c\rar_a c'$ as an element of the vector space $\cat{C}(a\cdot c,c')$ (using the notation from Notation \ref{ZXhomsetnotation} and Definition \ref{ZXtensoreddef}). Given another morphism 

\subsubsection{Tensor product of enriched tensored categories}

\begin{df}\label{ZXenrichtensorinit}
	The \emph{enriched cartesian product $\cat{C}\hat{\bxc{A}}\cat{D}$} of two $\cat{A}$-enriched categories $\cat{C}$ and $\cat{D}$ is the $\cat{A}$-enriched category whose objects are symbols $c\bxc{A}d$ with $c\in \cat{C}$ and $d\in \cat{D}$, and whose hom-objects are given by:
	\begin{equation}
	\underline{\cat{C}\hat{\bxc{A}}\cat{D}}\left(c\boxtimes d,c'\boxtimes d'\right)\colon =\underline{\cat{C}}(c,c')\otimes\underline{\cat{D}}(d,d'),
	\end{equation}
	where $\otimes$ is the tensor product in $\cat{A}$. Composition is given by first applying the braiding in $\cat{A}$ and then the compositions in $\cat{C}$ and $\cat{D}$.
\end{df}

Definition \ref{ZXenrichtensorinit} has an undesirable feature: if $\cat{C}$ and $\cat{D}$ from the above definition are semi-simple and idempotent complete, $\cat{C}\hat{\bxc{A}}\cat{D}$ in general will not be. Another, more prosaic, problem is that this notion of tensor product is not compatible with direct sums, we will fix this momentarily.

\begin{df}\label{ZXcauchycompl}
	The \emph{Cauchy completion} of a $\cat{A}$-enriched category $\cat{C}$ is the category with objects $n$-tuples of objects from $\cat{C}$ together with a matrix of morphisms from $\cat{C}$ that is idempotent as morphism from the $n$-tuple to itself. Morphisms are matrices of morphisms from $\cat{C}$ that commute with the idempotents.
\end{df}

\begin{rmk}
	Considering $n$-tuples of objects ensures compatibility with direct sums, picking idempotents ensures the category is idempotent complete. Note that $\cat{C}$ includes into its Cauchy completion. Any functor of $\cat{A}$-enriched categories induces a functor between the Cauchy completions, and the Cauchy completion of any category tensored over $\cat{A}$ is also tensored over $\cat{A}$.
\end{rmk}

\begin{df}\label{ZXaproddef}
	The \emph{$\cat{A}$-product $\cat{C}\bxc{A}\cat{D}$} of two $\cat{A}$-enriched categories $\cat{C}$ and $\cat{D}$ is the Cauchy completion of $\cat{C}\hat{\bxc{A}}\cat{D}$.
\end{df}

\begin{prop}\label{ZXcarttensored}
	If $\cat{C},\cat{D}\in\lcat{A}$, then $\cat{C}\bxc{A}\cat{D}$ is tensored over $\cat{A}$ with tensoring,
	\begin{equation}\label{ZXtensorontensor}
	a\cdot (c\bx d)\cong (a\cdot c)\bx d\cong c\bx (a\cdot d),
	\end{equation}
	and we have isomorphisms:
	$$
	(a\cdot c)\bx d\cong c\bx (a\cdot d).
	$$
\end{prop}
\begin{proof}
	For the first part, recall, from Definition \ref{ZXtensoreddef}, that it is enough to show that $c\bx (a\cdot d)$ satisfies:
	\begin{equation}\label{ZXproofgoal1}
	\underline{\cat{A}}(a,\underline{\cat{C}\bxc{A}\cat{D}}(c\bx d,c'\bx d'))\cong \underline{\cat{C}\bxc{A}\cat{D}}(c\bx (ad),c'\bx d').
	\end{equation}
	As $a\cdot (c\boxtimes d)$ is characterised by this equation, this will both establish existence of the tensor structure and $a\cdot (c\boxtimes d)\cong c\bx (a\cdot d)$.
	
	Substituting in the definition of the hom-objects in the $\cat{A}$-product, we see we are trying to find
	$$
	\underline{\cat{A}}(a,\underline{\cat{C}}(c,c')\otimes\underline{\cat{D}}(d,d'))\cong \underline{\cat{C}}(c,c')\otimes\underline{\cat{D}}(ad,d').
	$$
	Applying Lemma \ref{ZXaintohom} to $\cat{A}$ viewed as a category tensored over itself, we see that the left hand side reads:
	\begin{equation}\label{ZXsomestep}
	\underline{\cat{C}}(c,c')\cdot \underline{\cat{A}}(a,\underline{\cat{D}}(d,d'))\cong \underline{\cat{C}}(c,c')\otimes\underline{\cat{D}}(ad,d'),
	\end{equation}
	where the last isomorphism is (\ref{ZXtensenriched}). This gives us the desired isomorphism \eqref{ZXproofgoal1}.
	
	To establish the remaining assertion, observe that besides (\ref{ZXsomestep}) we also have, after applying the symmetry in $\cat{A}$
	\begin{align*}
	\underline{\cat{A}}(a,\underline{\cat{C}}(c,c')\otimes\underline{\cat{D}}(d,d'))&\cong \underline{\cat{D}}(d,d')\cdot \underline{\cat{A}}(a,\underline{\cat{C}}(c,c'))\\
	&\cong \underline{\cat{C}}(ac,c')\otimes\underline{\cat{D}}(d,d')\\
	&\cong \underline{\cat{C}\hat{\bxc{A}}\cat{D}}((ac)\bxc{A}d,c'\bxc{A}d'),
	\end{align*}
	where we used the symmetry in $\cat{A}$ again the penultimate line.
\end{proof}

The $\cat{A}$-product is symmetric in the sense that:

\begin{df}
	Let $\cat{C},\cat{D}\in\lcat{A}$, then the \emph{switch functor} $S\colon \cat{C}\bxc{A}\cat{D}\rar\cat{D}\bxc{A}\cat{C}$ is defined by
	$$
	c\bxc{A}d\mapsto d\bxc{A}c
	$$
	at the level of objects and
	$$
	\underline{\cat{C}}(c,c')\otimes \underline{\cat{D}}(d,d')\raru{s}\underline{\cat{D}}(d,d')\otimes\underline{\cat{C}}(c,c'),
	$$
	where $s$ is the symmetry in $\cat{A}$.
\end{df}

As the monoidal structure and the symmetry in $\cat{A}$ satisfy the appropriate coherence equations, the $\cat{A}$-product and the switch functor will strictly satisfy the coherence equations for a symmetric monoidal structure on the 2-category of categories enriched in and tensored over $\cat{A}$. That is, $(\lcat{A}, \bxc{A}, S)$ is a (strict) symmetric monoidal 2-category.

Given this $\cat{A}$-product, we can define:
\begin{df}\label{ZXatensordef}
	Let $\cat{C}$ be an $\cat{A}$-enriched category. Then a \emph{$\cat{A}$-tensor structure} is a pair of functors:
	$$
	\otimes\colon \cat{C}\bxc{A}\cat{C}\rar\cat{C},\quad\quad \mathbb{I}\colon \underline{\cat{A}}\rar \cat{C},
	$$
	equipped with associators and unitors satisfying the usual coherence conditions.
\end{df}

\begin{prop}\label{ZXunitenrcart}
	The unit for the enriched cartesian product of enriched and tensored categories is $\cat{A}$ enriched over itself, denoted by $\underline{\cat{A}}$.
\end{prop}

\subsubsection{Change of basis}\label{ZXchangeofbasissection}
Given monoidal categories $\cat{C}$ and $\cat{D}$, and a monoidal functor $F\colon \cat{C}\rar\cat{D}$, one gets a 2-functor $(F)_*$ from the 2-category of categories enriched over $\cat{C}$ to that of those enriched over $\cat{D}$, known as the ``change of basis'' functor. For a treatment of change of basis along monoidal functors, see \cite{Cruttwell2008}. The proofs from this reference translate straightforwardly to the lax monoidal case:

\begin{df}\label{ZXlaxmondef}
	A \emph{lax monoidal functor} from a monoidal category $\cat{C}$ to a monoidal category $\cat{D}$ is a functor $F\colon \cat{C}\rar \cat{D}$, together with a natural transformation:
	$$
	\mu\colon F(-)\otimes F(-)\Rightarrow F(-\otimes-),
	$$
	and a morphism
	$$
	\mu^0\colon \mathbb{I}_\cat{D}\rar F(\mathbb{I}_\cat{C}),
	$$
	that satisfy the compatibility conditions with the associators $\alpha_\cat{C}$ and $\alpha_\cat{D}$:
	\begin{center}
		\begin{tikzcd}
			F(c)(F(c')F(c'')) \arrow[r,"\mu"] \arrow[d,"\alpha_\cat{D}"] 	&	F(c)F(c'c'') \arrow[d,"\mu"] \\
			(F(c)F(c'))F(c'') \arrow[d,"\mu"]								&	F(c(c'c'')) \arrow[dl,"F(\alpha_\cat{C})"]\\
			F((cc')c'') &,
		\end{tikzcd}
	\end{center}
	for all $c,c',c''\in\cat{C}$, and compatibility with the unitors:
	\begin{center}
		\begin{tikzcd}
			\mathbb{I}_{\cat{D}}F(c)\arrow[r,"\mu_0"]\arrow[d,"\lambda_\cat{D}"]	&F(\mathbb{I}_\cat{C})F(c) \arrow[d,"\mu_{\mathbb{I},c}"]\\
			F(c)& F(\mathbb{I}_{\cat{C}}c)\arrow[l,"F(\lambda_\cat{C})"],			
		\end{tikzcd}
	\end{center}
	and a similar condition for the right unitors.
\end{df}

In this section, we will focus on this lax case. We will make use of the following well-known results.

\begin{prop}\label{ZXchangebasisobjects}
	Let $(F\colon \cat{C}\rar \cat{D},\mu)$ be a lax monoidal functor, and let $\cat{M}$ be a $\cat{C}$-enriched category. Then the category $F\cat{M}$ obtained from $\cat{M}$ by applying $F$ to the hom-objects is a $\cat{D}$-enriched category, with composition given by the image of the composition in $\cat{M}$ under $F$ and identity morphisms the image of the identity morphisms under $F$ precomposed with $\mu^0$.
\end{prop}

We will omit the proof of this statement. It turns out that a change of basis is a 2-functor:
\begin{prop}\label{ZXchangebasisispseudofunct}
	Let $(F\colon \cat{C}\rar \cat{D},\mu)$ be a lax monoidal functor, then the assignment $\cat{M}\mapsto F\cat{M}$ extends to a 2-functor from the 2-category of $\cat{C}$-enriched categories to the 2-category of $\cat{D}$-enriched categories.
\end{prop}

If the monoidal categories involved are braided, we can additionally ask for the lax monoidal functor to be braided:
\begin{df}\label{ZXbraidedlaxmon}
	Let $(F\colon \cat{C}\rar \cat{D},\mu)$ be a lax monoidal functor between braided (or symmetric) monoidal categories with braidings (or symmetries) $\beta^\cat{C}$ and $\beta^\cat{D}$, respectively. Then $F$ is called \emph{braided} (or \emph{symmetric}) if the following diagram
	\begin{center}
		\begin{tikzcd}
			F(c) F(c') \arrow[r,"\beta^\cat{D}"] \arrow[d,"\mu"] 	&	 F(c')F(c) \arrow[d,"\mu"]\\
			F(cc') \arrow[r,"F(\beta^\cat{C})"] 					& F(c'c)
		\end{tikzcd}
	\end{center}
	commutes for all $c,c'\in\cat{C}$.
\end{df}

As discussed in the previous section, if the enriching category is symmetric monoidal, there is a notion of enriched Cartesian product (Definition \ref{ZXenrichtensorinit}), and hence of enriched monoidal object. Change of basis along a symmetric lax monoidal functor preserves these monoidal objects. We give a proof as we will need a slight variation of this argument in this thesis. This fact is a consequence of the following.

\begin{lem}\label{ZXchangebasisproduct}
	Let $\cat{M}$ and $\cat{N}$ be $\cat{C}$-enriched categories, where $(\cat{C},s^\cat{C})$ is a symmetric monoidal category, and let $(F\colon \cat{C}\rar \cat{D},\mu)$ be a symmetric lax monoidal functor, and denote the symmetry in $\cat{D}$ by $s^\cat{D}$. Then the assignment:
	$$
	M\colon F\cat{M}\cattens{\cat{D}}F\cat{N}\rar F(\cat{M}\cattens{\cat{C}}\cat{N}),
	$$
	which is the idenity on objects and $\mu$ on hom-objects, is a $\cat{D}$-enriched functor.
\end{lem}

\begin{proof}
	We need to check that $M$ preserves composition, this translates into checking that the outside of following diagram commutes:
	\begin{center}
		\begin{tikzcd}
			F\cat{M}F\cat{N}F\cat{M}F\cat{N}	\arrow[r,"\mu\mu"] \arrow[d,"s^\cat{D}"]		&	F(\cat{M}\cat{N})F(\cat{M}\cat{N}) \arrow[d,"\mu"]\\
			F\cat{M}F\cat{M}F\cat{N}F\cat{N}	\arrow[d,"\mu \mu"]							&	F(\cat{M}\cat{N}\cat{M}\cat{N}) \arrow[d,"F(s^\cat{D})"]\\
			F(\cat{M}\cat{M})F(\cat{N}\cat{N})	\arrow[r,"\mu"] \arrow[d,"F(\circ)F(\circ)"]&	F(\cat{M}\cat{M}\cat{N}\cat{N}) \arrow[d,"F(\circ \circ)"]\\
			F(\cat{M})F(\cat{N})				\arrow[r,"\mu"]								&	F(\cat{M}\cat{N}),
		\end{tikzcd}
	\end{center}
	where we have suppressed the objects in for example $F\cat{M}(m,m')$ from the notation for readability. The bottom square commutes by naturality of $\mu$, for the top square, we observe that the compatibility of $\mu$ with the associators (Definition \ref{ZXlaxmondef}) allows us to rewrite this as:
	\begin{center}
		\begin{tikzcd}
			F\cat{M}F\cat{N}F\cat{M}F\cat{N}	\arrow[r,"\mu"] \arrow[d,"s^\cat{D}"]	&	F\cat{M}F(\cat{N}\cat{M})F\cat{N} \arrow[d,"F(s^\cat{C})"]\arrow[r,"\mu\circ\mu"]	&	F(\cat{M}\cat{N}\cat{M}\cat{N})\arrow[d,"F(s^\cat{C})"]\\
			F\cat{M}F\cat{M}F\cat{N}F\cat{N}		\arrow[r,"\mu"]						&	F\cat{M}F(\cat{M}\cat{N})F\cat{N}	\arrow[r,"\mu\circ \mu"]						&	F(\cat{M}\cat{M}\cat{N}\cat{N}),
		\end{tikzcd}
	\end{center}
	where the rightmost square commutes by naturality of $\mu$, and the leftmost square is exactly the one from Defintion \ref{ZXbraidedlaxmon}.
\end{proof}

We observe that when $F$ is strong monoidal (so $\mu$ is an isomorphism), change of basis along $F$ takes the $\cat{C}$-enriched Cartesian product to the $\cat{D}$-enriched Cartesian product.

\begin{prop}\label{ZXmoncatpres}
	Let $F\colon \cat{C}\rar \cat{D}$ and $\cat{M}$ be as in the previous lemma. Assume further that $\cat{C}$ and $\cat{D}$ are symmetric, and that $F$ preserves the symmetry. Then, if $\cat{M}$ is $\cat{C}$-monoidal with monoidal structure $\otimes$, $F\cat{M}$ is $\cat{D}$ monoidal, with monoidal structure given by the composite
	$$
	F\cat{M}\cattens{\cat{D}}F\cat{M}\xrightarrow{M}F(\cat{M}\cattens{\cat{C}}\cat{M})\xrightarrow{F(\otimes)}F\cat{M}.
	$$
\end{prop}

\begin{proof}
	The monoidal structure is clearly functorial, as it is a composite of $\cat{D}$-enriched functors. As $\mu$ respects the associators for $\cat{C}$, $F$ will take the associators for $\cat{M}$ to associators for $F\cat{M}$, and similar for the unitors.
\end{proof}

This extends to:
\begin{prop}\label{ZXsymlaxinducessymspseudo}
	Let $F\colon \cat{C}\rar \cat{D}$ be a symmetric lax monoidal functor. Then the assignment $\cat{M}\mapsto F\cat{M}$ extends to a symmetric monoidal 2-functor from the 2-category of $\cat{C}$-enriched categories, with enriched cartesian product, to the 2-category of $\cat{D}$-enriched categories, with enriched cartesian product.
\end{prop}

In fact for a given $\cat{C},\cat{D}$ monoidal categories and $\cat{C}$-enriched category $\cat{M}$, ``change of basis'' $(-)_*$ is itself a functor from the functor category $\mathbf{MonCat}^L(\cat{C},\cat{D})$ of lax monoidal functors from $\cat{C}$ to $\cat{D}$ and their natural transformations to the category of $\cat{D}$-enriched categories. We remind the reader of the following definition:

\begin{df}
	Let $(F,\mu)$ and $(G,\nu)$ be lax monoidal functors between $\cat{C}$ and $\cat{D}$, then a \emph{lax monoidal natural transformation} $\sigma\colon F\Rightarrow G$ is a natural transformation such that for all $c,c'\in\cat{C}$ the following diagram commutes:
	\begin{center}
		\begin{tikzcd}
			F(c)F(c') \arrow[r, "\mu_{c,c'}"] \arrow[d,"\sigma_c\sigma_{c'}"] 	&	F(cc')	\arrow[d,"\sigma_{cc'}"]\\
			G(c)G(c') \arrow[r,"\nu_{c,c'}"] 									&	G(cc').
		\end{tikzcd}
	\end{center}
	
\end{df}

\begin{prop}\label{ZXlaxnattrafotofunctor}
	Let $F,G\colon \cat{C}\rar \cat{D}$ be lax monoidal functors and $\cat{M}$ be $\cat{C}$-enriched, and let $\sigma\colon F\Rightarrow G$ be a lax monoidal natural transformation. Then, for every $\cat{C}$-enriched category $\cat{M}$, we have a $\cat{D}$-enriched functor
	$$
	\Sigma\colon  F\cat{M} \rar G\cat{M},
	$$
	given by the identity on objects and $\sigma$ on the hom-objects. Furthermore, the assignment $\sigma\mapsto\Sigma$ preserves composition of natural transformations.
\end{prop}

As being tensored is a property of the enrichment, the following is automatic.

\begin{prop}\label{ZXbasechtens}
	Let $\cat{M}$ be enriched and tensored over a monoidal category $\cat{C}$, and let $F\colon \cat{C}\rar\cat{D}$ be a lax monoidal functor. Then $F\cat{M}$ is enriched and tensored over $\cat{D}$.
\end{prop}


\begin{thebibliography}{BDSPV15}
	
	\bibitem[BDSPV15]{Bartlett2015a}
	Bruce Bartlett, Christopher~L. Douglas, Christopher~J. Schommer-Pries, and
	Jamie Vicary.
	\newblock {Modular categories as representations of the 3-dimensional bordism
		2-category}.
	\newblock {\em arXiv preprint arXiv:1509.06811}, 2015.
	
	\bibitem[BK01]{Bakalov2001a}
	B~Bakalov and AA~Kirillov.
	\newblock {\em {Lectures on tensor categories and modular functors}}.
	\newblock American Mathematical Society, 2001.
	
	\bibitem[Cru08]{Cruttwell2008}
	G~S~H Cruttwell.
	\newblock {\em {Normed Spaces and the Change of Base for Enriched Categories}}.
	\newblock PhD thesis, Dalhousie University, 2008.
	
	\bibitem[Del90]{Deligne1990}
	Pierre Deligne.
	\newblock {Cat{\'{e}}gories tannakiennes}.
	\newblock {\em The Grothendieck Festschrift}, II(87):111--195, 1990.
	
	\bibitem[Del02]{Deligne2002}
	Pierre Deligne.
	\newblock {Cat{\'{e}}gories tensorielles}.
	\newblock {\em Moscow Mathematical Journal}, 2(2):227--248, 2002.
	
	\bibitem[DGNO10]{Drinfeld2009}
	Vladimir Drinfeld, Shlomo Gelaki, Dmitri Nikshych, and Victor Ostrik.
	\newblock {On braided fusion categories I}.
	\newblock {\em Selecta Mathematica}, 16(1):1--119, 2010.
	
	\bibitem[ENO05]{Etingof2002}
	Pavel Etingof, Dmitri Nikshych, and Viktor Ostrik.
	\newblock {On fusion categories}.
	\newblock {\em Annals of Mathematics}, 162(2):581--642, 2005.
	
	\bibitem[JS86]{Joyal1986}
	Andr{\'{e}} Joyal and Ross Street.
	\newblock {Braided Monoidal Categories}.
	\newblock {\em Macquarie Math. Reports}, (860081), 1986.
	
	\bibitem[{Mac}71]{MacLane}
	Saunders {Mac Lane}.
	\newblock {\em {Categories for the Working Mathematician}}.
	\newblock Springer-Verlag, New York, NY, second edition, 1971.
	
	\bibitem[M{\"{u}}g10]{Turaev2010a}
	Michael M{\"{u}}ger.
	\newblock {On the structure of braided crossed G-categories}.
	\newblock In Vladimir~G. Turaev, editor, {\em Homotopy Quantum Field Theory},
	chapter Appendix 5, pages 221--235. European Mathematical Society, 2010.
	
	\bibitem[Tur10]{Turaev2010}
	Vladimir~G. Turaev.
	\newblock {\em {Homotopy Quantum Field Theory}}.
	\newblock European Mathematical Society, 2010.
	
	\bibitem[Was17a]{Wasserman2017a}
	Thomas~A. Wasserman.
	\newblock {The Drinfeld Centre of a Symmetric Fusion Category is 2-Fold
		Monoidal}.
	\newblock {\em arXiv preprint arXiv:1711.03394}, 2017.
	
	\bibitem[Was17b]{Wasserman2017}
	Thomas~A. Wasserman.
	\newblock {The Symmetric Tensor Product on the Drinfeld Centre of a Symmetric
		Fusion Category}.
	\newblock {\em arXiv preprint arXiv: 1710.06461}, 2017.
	
	\bibitem[Was19]{Wassermanc}
	Thomas~A. Wasserman.
	\newblock {The Reduced Tensor Product of Braided Fusion Categories containing a
		Symmetric Subcategory}.
	\newblock {\em Preprint, available on author's webpage}, 2019.
	
\end{thebibliography}
\end{document}